\documentclass{amsart}
\numberwithin{equation}{section}
\newtheorem{theorem}{Theorem}[section]
\newtheorem{corollary}[theorem]{Corollary}
\newtheorem{proposition}[theorem]{Proposition}
\newtheorem{conjecture}[theorem]{Conjecture}
\newtheorem{lemma}[theorem]{Lemma}
\theoremstyle{definition}
\newtheorem{definition}[theorem]{Definition}

\newtheorem{example}[theorem]{Example}

\theoremstyle{remark}
\newtheorem{remark}[theorem]{\bf\em Remark}

\usepackage{graphicx}
\usepackage{amsfonts}
\usepackage{amssymb}
\usepackage{amsmath}
\usepackage{color}
\usepackage[inner=1in,outer=1in,bottom=1in,top=1in]{geometry}
\DeclareMathOperator{\lcm}{lcm}
\DeclareMathOperator{\ord}{ord}

\begin{document}
\title[Modular periodicity of exponential sums of symmetric Boolean functions]{Modular periodicity of exponential sums of symmetric Boolean functions and some of its consequences}

\author{Francis N. Castro}
\address{Department of Mathematics, University of Puerto Rico, San Juan, PR 00931}
\email{franciscastr@gmail.com}

\author{Luis A. Medina}
\address{Department of Mathematics, University of Puerto Rico, San Juan, PR 00931}
\email{luis.medina17@upr.edu}

% Abstract  
\begin{abstract}
This work brings techniques from the theory of recurrent integer sequences to the problem of balancedness of symmetric Booleans functions.  In particular, the periodicity modulo $p$ ($p$ odd prime) of exponential sums of symmetric Boolean functions is considered.  Periods modulo $p$, bounds for periods and relations between them are obtained for these exponential sums.  The concept of avoiding primes is also introduced.  This concept and the bounds presented in this work are used to show that some classes of symmetric 
Boolean functions are not balanced.  In particular, every elementary symmetric Boolean function of degree not a power of 2 and less than 2048 is not balanced.  For instance, the 
elementary symmetric Boolean function in $n$ variables of degree $1292$ is not balanced because the prime $p = 176129$ does not divide its exponential sum for any positive 
integer $n$.  It is showed that for some symmetric Boolean functions, the set of primes avoided by the sequence of exponential sums contains a subset that has positive density within 
the set of primes.  Finally, in the last section, a brief study for the set of primes that divide some term of the sequence of exponential sums is presented.
\end{abstract}

% General info
\subjclass[2010]{05E05, 11T23, 11B50}
\date{\today}
% Keywords
\keywords{Symmetric Boolean functions, exponential sums, recurrences, Pisano periods}

\maketitle
%%%%%%%%%%%%%%%%%%%%%%%%%%%%%%%%%%%%%%%%%%%%%%
% Section: Introduction
%%%%%%%%%%%%%%%%%%%%%%%%%%%%%%%%%%%%%%%%%%%%%%
\section{Introduction}
Boolean functions are beautiful combinatorial objects with applications to many areas of mathematics as well as outside the field.  Some examples include combinatorics, electrical engineering, game theory, 
the theory of error-correcting codes, and cryptography.   In the modern era, efficient implementations of Boolean functions with many variables is a challenging problem due to memory restrictions of current 
technology. Because of this, symmetric Boolean functions are good candidates for efficient implementations. 

In many applications, especially ones related to cryptography, it is important for Boolean functions to be balanced. Every symmetric function is a combination of elementary symmetric polynomials, thus an 
important step should be to understand the balancedness of them. In \cite{cusick2}, Cusick, Li and St$\check{\mbox{a}}$nic$\check{\mbox{a}}$ proposed a conjecture that explicitly states when an elementary 
symmetric function is balanced:\medskip

\noindent
{\it There are no nonlinear balanced elementary symmetric Boolean functions except for degree $k=2^{l}$ and $2^{l+1}D-1$-variables, where $l,D$ are positive integers. }\medskip

\noindent
Surprisingly, this conjecture is still open, but some advances have been made.  For some history of the problem, as well as its current state, the reader is invited to read \cite{cgm1, cm1, cm2, cusick2, 
cusick1, gao, ggz, pott}.

The subject of Boolean functions can be studied from the point of view of complexity theory or from the algebraic point of view as it is done in this article, where the periodicity of exponential sums of 
symmetric Boolean functions modulo a prime is exploited. The study of divisibility properties of Boolean functions is not new.  In fact, it is an active area of research \cite{sperber, ax, morenocastro, 
mm1, mm, mcsk}.  However, the authors believe that the modular periodicity of these functions has not been study in detail nor its possible connections to Cusick-Li-St$\check{\mbox{a}}$nic$\check{\mbox{a}}$'s 
conjecture.

In \cite{cm1}, Castro and Medina viewed exponential sums of symmetric Boolean functions as integer sequences.  As part of their study, they showed that these sequences satisfy homogenous linear recurrences 
with integer coefficients.  Moreover, in the case of one elementary symmetric function, they were able to provide the minimal homogenous linear recurrence.  It is a well-established result in number theory 
that recurrent integer sequences are periodic or eventually periodic modulo an integer $m$, with the first serious study being done by Lucas \cite{lucas}. Some of the results available on this topic are used 
in this manuscript to find bounds for the periods of these sequences.  These bounds and the periodicity of exponential sums of symmetric Boolean functions are used to show that some families are not balanced.
For example, every elementary symmetric Boolean function of degree not a power of 2 and less than 2048 is not balanced.  In particular, the elementary symmetric Boolean function in $n$ variables of degree 
1292 is not balanced because the prime $p=176129$ does not divide its exponential sum for any positive integer $n$.  One of the main goals of this manuscript is to provide some insights about the 
$p$-divisibility ($p$ prime) of the exponential sum of symmetric Boolean functions.

This work is divided in various parts.  It starts with some preliminaries (section \ref{prelim}) about symmetric Boolean functions.   It follows with a review of the periodicity modulo $m$ of linear 
recurrences (section \ref{periodicitysec}). This is done because, to the knowledge of the authors, it is not common to find the subjects of Boolean functions and periodicity modulo $m$ of recurrent 
sequences together in a manuscript. Section \ref{periodicitysec} also contains Theorem \ref{weakperiodthm} (Vince \cite{vince}), which is an important tool for finding upper bounds for the periods of the 
sequences considered in this article.  After that, in section \ref{boundsec}, the periodicity modulo $p$ ($p$ an odd prime) of exponential sums of symmetric Boolean functions is studied in more detail.  
In particular, the reader can find bounds and relations for these periods.  In section \ref{avoidprimes}, the concept of avoiding primes is introduced.  This concept and the bounds presented in section 
\ref{boundsec} are used to show that some of these families are not balanced.  Moreover, it is showed that for some symmetric Boolean functions, the set of primes avoided by the sequence of exponential 
sums has positive density within the set of primes.  Finally, in section \ref{sectionP}, a small study for the set of primes that divide some term of the sequence of exponential sums is presented.

%%%%%%%%%%%%%%%%%%%%%%%%%%%%%%%%%%%%%%%%%%%%%%%%%%%%%%%%%%%%
% Section: Preliminaries 
%%%%%%%%%%%%%%%%%%%%%%%%%%%%%%%%%%%%%%%%%%%%%%%%%%%%%%%%%%%%
\section{Preliminaries}
\label{prelim}
Let $\mathbb{F}_2$ be the binary field, $\mathbb{F}_2^{\,n} = \{(x_1,\ldots, x_n) | x_i \in \mathbb{F}_2, i = 1, . . . , n\}$, and $F(X) = F(X_1, \ldots,X_n)$ be a polynomial in $n$ variables over 
$\mathbb{F}_2$. The exponential sum associated to $F$ over $\mathbb{F}_2$ is
\begin{equation}
S(F)=\sum_{x_1,\ldots,x_n\in \mathbb{F}_2} (-1)^{F(x_1,\ldots,x_n)}.
\end{equation}
A Boolean function $F$ is called balanced if $S(F) = 0$, i.e. the number of zeros and the number of ones are equal in the truth table of $F$. This property is important for some applications in 
cryptography.

Any symmetric Boolean function is a linear combination of elementary symmetric polynomials.   Let $\sigma_{n,k}$ be the elementary symmetric polynomial in $n$ variables of degree $k$. For example,
\begin{equation}
\sigma_{4,3} = X_1 X_2 X_3+X_1 X_4 X_3+X_2 X_4 X_3+X_1 X_2 X_4.
\end{equation}
Then, every symmetric Boolean function can be identified with an expression of the form
\begin{equation}
\label{genboolsym}
\sigma_{n,k_1}+\sigma_{n,k_2}+\cdots+\sigma_{n,k_s},
\end{equation}
where $1\leq k_1<k_2<\cdots<k_s$ are integers.  For the sake of simplicity, the notation $\sigma_{n,[k_1,\cdots,k_s]}$ is used to denote (\ref{genboolsym}).  For example,
\begin{eqnarray}
\sigma_{3,[2,1]}&=&\sigma_{3,2}+\sigma_{3,1}\\ \nonumber
&=& X_1 X_2+X_3 X_2+X_1 X_3+X_1+X_2+X_3.
\end{eqnarray}
It is not hard to show that if $1\leq k_1 < k_2 < \cdots < k_s$ are fixed integers, then
\begin{equation}
\label{maingen}
S(\sigma_{n,[k_1,k_2,\cdots,k_s]}) =\sum_{j=0}^n (-1)^{\binom{j}{k_1}+\binom{j}{k_2}+\cdots+\binom{j}{k_s}}\binom{n}{j}.
\end{equation}
\begin{remark}
Observe that the right hand side of (\ref{maingen}) makes sense for $n\geq 1$, while the left hand side exists for $n\geq k_s$. Throughout the rest of the article, $S(\sigma_{n,[k_1,k_2,\cdots,k_s]})$ 
should be interpreted as the expression on the right hand side, so it makes sense to talk about ``exponential sums" of symmetric Boolean functions with less variables than their degrees. 
\end{remark}

Equation (\ref{maingen}) links the problem of balancedness of $\sigma_{n,[k_1,\cdots,k_s]}$ to the problem of bisecting binomial coefficients (this was first discussed by Mitchell \cite{mitchell}). 
A solution $(\delta_0,\delta_1,\cdots, \delta_n)$ to the equation
\begin{equation}
\label{bisec}
 \sum_{j=0}^n x_j \binom{n}{j}=0,\,\,\, x_j \in \{-1,1\},
\end{equation}
is said to give a {\em bisection of the binomial coefficients} $\binom{n}{j}$, $0\leq j \leq n.$  Observe that a solution to (\ref{bisec}) provides us with two disjoints sets $A,B$ such that 
$A\cup B =\{0,1,2,\cdots,n\}$ and
\begin{equation}
 \sum_{j \in A} \binom{n}{j}=\sum_{j\in B}\binom{n}{j}=2^{n-1}.
\end{equation}
If $n$ is even, then $\delta_j = \pm(-1)^j$, for $j=0,1,\cdots, n$, are two solutions to (\ref{bisec}).  On the other hand,
if $n$ is odd, then the symmetry of the binomial coefficients  implies that $(\delta_0,\cdots, \delta_{(n-1)/2},-\delta_{(n-1)/2},\cdots,-\delta_0)$  are $2^{(n+1)/2}$ solutions to (\ref{bisec}). These are called trivial solutions.  A balanced symmetric Boolean function in $n$ variables which corresponds to one of the trivial solutions of (\ref{bisec}) is said to be a {\it trivially balanced function}.  
Computations suggest that a majority of the balanced symmetric Boolean functions are trivially balanced, thus it is of great interest to find non-trivially balanced symmetric Boolean functions.  In 
the literature, these functions are called {\it sporadic} balanced symmetric Boolean functions, see \cite{cusick3, cusick2, jeff} for more information.

In \cite{cm1}, Castro and Medina used (\ref{maingen}) to study exponential sums of  symmetric polynomials from the point of view of integer sequences.  As part of their study, they showed that the sequence 
$\{S(\sigma_{n,[k_1,\cdots,k_s]})\}_{n\in \mathbb{N}}$ satisfies the homogeneous linear recurrence 
\begin{equation}
\label{mainrec}
x_n=\sum_{j=1}^{2^r-1}(-1)^{j-1}\binom{2^r}{j}x_{n-j},
\end{equation}
where $r=\lfloor\log_2(k_s)\rfloor+1$ (this result also follows from \cite[Th. 3.1, p. 248]{cai})  and used this result to compute the asymptotic behavior $S(\sigma_{n,[k_1,\cdots,k_s]})$ as $n\to \infty$.  
To be specific,
\begin{equation}
\label{asymplimit}
\lim_{n\to\infty}\frac{1}{2^n}S(\sigma_{n,[k_1,\cdots,k_s]}) = c_0(k_1,\cdots, k_s)
\end{equation}
where 
\begin{equation}
c_0(k_1,\cdots,k_s) = \frac{1}{2^r} \sum_{j=0}^{2^r-1}(-1)^{\binom{j}{k_1}+\cdots+\binom{j}{k_s}}.
\end{equation}
Limit (\ref{asymplimit}) gives rise to the concept of asymptotically balanced symmetric Boolean function, which was also introduced in \cite{cm1}.  A symmetric Boolean function $\sigma_{n,[k_1,\cdots,k_s]}$ 
is asymptotically balanced if $c_0(k_1,\cdots,k_s)=0.$  They used this concept to show that Cusick-Li-St$\check{\mbox{a}}$nic$\check{\mbox{a}}$'s conjecture is true asymptotically (this result was 
recently re-established in \cite{ggz}).    See \cite{cm1} for more details.

In \cite{cm2}, Castro and Medina extended many of the results presented in \cite{cm1} to some perturbations of symmetric Boolean functions.  Recall that $\sigma_{n,k}$ is the elementary symmetric polynomial 
of degree $k$ in the variables $X_1,\cdots, X_n$.  Suppose that $1\leq j<n$ and let $F({\bf X})$ be a binary polynomial in the variables $X_1,\cdots, X_j$ (the first $j$ variables in $X_1,\cdots, X_n$).  
Castro and Medina showed that the exponential sum of the perturbation $\sigma_{n,[k_1,\cdots,k_s]}+F({\bf X})$ satisfies the following relation
\begin{equation}
\label{perturbationeq}
S(\sigma_{n,[k_1,\cdots,k_s]}+F({\bf X}))=\sum_{m=0}^j C_m(F)S\left(\sum_{i=0}^m\binom{m}{i}(\sigma_{n-j,[k_1-i,\cdots,k_s-i}])\right),
\end{equation}
where
\begin{equation}
C_m(F)=\sum_{{\bf x}\in \mathbb{F}_2 \text{ with }w_2({\bf x})=m}(-1)^{F({\bf x})}.
\end{equation}
\begin{remark}
There are three things to observe about equation (\ref{perturbationeq}). First, it is clear that the value of $\binom{m}{i}$ that is inside the exponential sum can be taken mod 2, since only the parity 
matters.  Second, if $k_l-i <0$, then the term $\sigma_{n-j,k_l-i}$ does not exist and so it is not present in the equation.  Finally, in the case that $k_l-i=0$, the elementary polynomial $\sigma_{n-j,0}$ 
should be interpreted as 1.
\end{remark}

Note that equation (\ref{perturbationeq}) implies that these type of perturbations also satisfy recurrence (\ref{mainrec}) and therefore many of the results presented in this article about the periodicity 
modulo a prime of exponential sums of symmetric Boolean functions also apply to them.  In fact, later in section \ref{avoidprimes} these perturbations are examined in the context of sequences avoiding 
primes.

The periodicity modulo an integer of a recurrent sequence is closely related to its characteristic polynomial (as expected).  The main focus of this manuscript is the periodicity modulo $p$ of exponential 
sums of symmetric Boolean functions, which, as mentioned before, satisfy recurrence (\ref{mainrec}).  Observe that the characteristic polynomial of (\ref{mainrec}) is given by
\begin{equation}
(t-2)\Phi_4(t-1)\Phi_8(t-1)\cdots\Phi_{2^r}(t-1),
\end{equation}
where $\Phi_n(t)$ represents the $n$-th cyclotomic polynomial.  This factorization is crucial and it is used in section \ref{boundsec} to find bounds for the Pisano period of 
\begin{equation}
\label{seqtoconsider}
\{S(\sigma_{n,[k_1,\cdots,k_s]})\mod p\}_{n\in \mathbb{N}},
\end{equation}
when $p$ is a prime.   This factorization also implies that sequence (\ref{seqtoconsider}) is periodic for every odd prime $p$.

Naturally, a sequence may satisfy (\ref{mainrec}) but have a minimal linear recurrence with integer coefficients different from it.  In the case of the elementary symmetric polynomial, Castro and Medina 
were able to find the minimal homogenous linear recurrence with integer coefficients that $\{S(\sigma_{n,k})\}$ satisfies.  To be specific, let $\epsilon(n)$ be defined as
\begin{equation}
\epsilon(n)=\left\{\begin{array}{cl}
    0, & \text{if }n \text{ is a power of 2,}  \\
    1, & \text{otherwise.}
       \end{array}\right.
\end{equation}
Then, the following result holds.
\begin{theorem}
\label{charpolyn}
Let $k$ be a natural number and $\chi_k(t)$ be the characteristic polynomial associated to the minimal linear recurrence with integer coefficients that $\{S(\sigma_{n,k})\}_{n\in\mathbb{N}}$ satisfies.  
Let $\bar{k}=2\lfloor k/2\rfloor+1$.  Express $\bar{k}$ as its $2$-adic expansion 
\begin{equation}
\bar{k}=1+2^{a_1}+2^{a_2}+\cdots+2^{a_l},
\end{equation}
where the last exponent is given by $a_l=\lfloor \log_2(\bar{k})\rfloor.$  Then,
\begin{equation}
\chi_k(t)=(t-2)^{\epsilon(k)}\prod_{j=1}^l \Phi_{2^{a_j+1}}(t-1).
\end{equation}
In particular, the degree of the minimal linear recurrence that $\{S(\sigma_{n,k})\}_{n\in\mathbb{N}}$ satisfies is equal to $2\lfloor k/2 \rfloor + \epsilon(k)$.
\end{theorem}
\noindent

A generalization of Theorem \ref{charpolyn} for the case of a general symmetric Boolean function also appears in \cite{cm1}, but such generalization is not needed for the work presented in this article, 
thus the authors decided not to include it.   What is important, however, is the fact that if $\chi_{k_1,\cdots,k_s}(t)$ is the characteristic polynomial associated to the minimal homogenous linear 
recurrence that $\{S(\sigma_{n,[k_1,\cdots,k_s]})\}$ satisfies, then $\Phi_{2^r}(t-1)$, with $r=\lfloor \log_2(k_s) \rfloor+1$, is always a factor of $\chi_{k_1,\cdots,k_s}(t)$.  This is used in 
section \ref{boundsec} to find lower bounds for the period of (\ref{seqtoconsider}).

%Recurrence (\ref{mainrec}) implies that the sequence $\{S(\sigma_{n,k_1}+\cdots+\sigma_{n,k_s})\mod p\}_{n\in \mathbb{N}}$ is periodic for every odd prime $p$.

In the next section, a brief introduction to the periodicity modulo an integer of linear recurrences is presented. The expert reader may skip the majority of it, however he/she is encouraged to review 
Theorem \ref{weakperiodthm}, as it is a result used in later sections.  Later, in section \ref{boundsec}, the period mod $p$ of $\{S(\sigma_{n,[k_1,\cdots,k_s]})\}$ is studied in more detail.

%%%%%%%%%%%%%%%%%%%%%%%%%%%%%%%%%%%%%%%%%%%%%%%%%%%%%%%%%%%%
% Section: Periodicity mod m of linear recurrences
%%%%%%%%%%%%%%%%%%%%%%%%%%%%%%%%%%%%%%%%%%%%%%%%%%%%%%%%%%%%
\section{Periodicity mod $m$ of linear recurrences}
\label{periodicitysec}
As mentioned in the introduction, it is a well-established result that integer sequences satisfying homogenous linear recurrences with integer coefficients are periodic or eventually periodic modulo an 
integer $m$.   Below is a review of this fact.

Suppose that the integer sequence $\{x_n\}$ satisfies the linear recurrence 
\begin{equation}
x_n=a_1 x_{n-1}+a_2 x_{n-2}+\cdots+ a_r x_{n-r}, \,\, \text{ for }n\geq r,
\end{equation}
with $a_1,\cdots, a_r \in \mathbb{Z}$ and initial conditions $x_0,\cdots, x_{r-1}$.  Let 
\begin{equation}
\label{matrixeq}
A=\left(
\begin{array}{ccccc}
 0 & 1 &  0 & \cdots & 0 \\
 0 & 0 & 1  & \cdots & 0\\
 \vdots & \vdots & \vdots & \ddots & \vdots \\
 0 & 0 & 0 & \cdots & 1 \\
 a_r & a_{r-1} & a_{r-2}& \cdots & a_1
\end{array}
\right) \,\, \text{ and }\,\, 
X_n =\left(
\begin{array}{c}
x_{n}\\
x_{n+1}\\
\vdots\\
x_{n+r-1}
\end{array}
\right).
\end{equation}
It is clear that $X_{n+1}=AX_n$, thus iteration leads to 
\begin{equation}
\label{matrixiteration}
X_{n}=A^nX_0.
\end{equation}  
The $r\times r$ matrix $A$ is called the {\it companion matrix} of the recurrence. The vector $X_0$ is called the {\it initial valued vector}.  The initial valued vector is simply the vector whose entries 
are the initial conditions of the linear recurrence.  It is assumed throughout the article that $X_0$ is not the zero vector.

Equation (\ref{matrixiteration}) links the periodicity of $\{x_n \mod m\}$ to the periodicity of $\{A^n \mod m\}$.  Consider the list
\begin{equation}
\label{listtomod}
I, A, A^2, A^3,\cdots, A^n,\cdots
\end{equation}
where each $A^i$ has been reduced modulo $m$.  As a list, (\ref{listtomod}) is infinite.  However, when viewed as a set, it is finite because the set of all $r\times r$ matrices with entries from the 
group $\mathbb{Z}/m\mathbb{Z}$ is a finite set.  Therefore, the Pigeonhole Principle implies the existence of a pair of integers $n$ and $k$ such that 
\begin{equation}
\label{eventuallyperiodic}
A^{n+k} \equiv A^k \mod m
\end{equation} 
and $n+k>k\geq 0$.   Invoke the Well-Ordering Principle to conclude that $\{A^n \mod m\}$, and thus $\{x_n \mod m\}$, is eventually periodic.

In the case when $\det(A)$ is relatively prime to $m$, the matrix $A$ is invertible mod $m$.  In view of (\ref{eventuallyperiodic}), this means that there is a positive integer $n$ such that 
$A^n \equiv I \mod m$, which is equivalent to saying that the sequence $\{x_n \mod m\}$ is periodic.   The least possible $n$ such that 
\begin{equation}
\label{period}
x_{n+k}\equiv x_k \mod m
\end{equation} 
for all $k\geq 0$ is called the {\em Pisano period} of the sequence and it is usually denoted by $\pi(m)$ if the context of the sequence is clear.   Also, it is clear that if $n_0$ is the order of 
$A$ mod $m$, then  $n_0$ satisfies (\ref{period}).  Therefore, $\pi(m)$ divides the order of the companion matrix modulo $m$.  The order of $A$ mod $m$ is called the {\em weak Pisano period} and is 
denoted by $\pi^*(m)$.  

Continue with the case $\gcd(\det(A),m)=1$.  Suppose that $X_0$ is not the zero vector modulo $m$.  The least possible integer $n$ such that
\begin{equation}
\label{restricted}
X_{n+k} \equiv s X_k \mod m,
\end{equation}
for all $k\geq 0$ where $s$ is some integer, is called the {\it restricted period} of the sequence $\{x_n \mod m\}$.  The restricted period is usually denoted by $\alpha(m)$.  Observe that the existence 
of $\alpha(m)$ follows from the fact that the sequence is periodic.  To be specific, the periodicity of $\{x_n \mod m\}$ implies
\begin{equation}
\label{restrictedexists}
X_{\pi(m)+k} \equiv X_k \mod m
\end{equation}
for all $k\geq 0$.  Therefore, the set of positive integers satisfying (\ref{restricted}) is not empty ($\pi(m)$ belongs to this set).  The integer $0<s(m)<m$ such that
\begin{equation}
X_{\alpha(m)+k}\equiv s(m)X_k \mod m,
\end{equation}
for all $k\geq 0$ is called the {\it multiplier} of the sequence $\{x_n \mod m\}$.  

It is not hard to show that $X_{n+k}\equiv s X_k\mod m$, for all $k\geq 0$ and some integer $s$, if and only if $\alpha(m)\,|\, n$.  Therefore, not only does (\ref{restrictedexists}) imply the existence 
of $\alpha(m)$, it also implies that $\alpha(m)$ divides $\pi(m)$.  Write $\pi(m)=\alpha(m)l(m)$.  Observe that
\begin{equation}
X_0 \equiv X_{\pi(m)} \equiv X_{\alpha(m)l(m)} \equiv s(m)^{l(m)} X_0 \mod m.
\end{equation}
Thus, $s(m)$ is a unit modulo $m$.  The order of $s(m)$ is denoted by $\beta(m)$.  Note that
\begin{equation}
\label{betadivides}
\beta(m)\,|\, l(m)=\frac{\pi(m)}{\alpha(m)}.
\end{equation}
On the other hand, since $\beta(m)$ is the order of the multiplier modulo $m$, then
\begin{equation}
\label{pidivides}
X_{\alpha(m)\beta(m)+k}\equiv s(m)^{\beta(m)}X_k\equiv X_k \mod m,
\end{equation}
for all $k\geq 0$.  This implies that $\pi(m)$ divides $\alpha(m)\beta(m)$.  Together, (\ref{betadivides}) and (\ref{pidivides}) yield $\pi(m)=\alpha(m)\beta(m)$.

As it was the case for the Pisano period, one can define the corresponding {\it weak restricted period}.  The natural way to define it is as the least positive integer $n$ such that
\begin{equation}
A^n \equiv s I \mod m,
\end{equation}
for some integer $s$.  The weak restricted period is denoted by $\alpha^*(m)$.  The numbers $s^*(m)$ and $\beta^*(m)$ are defined in the usual sense.  It is not hard to show that
{ \renewcommand\labelenumi{(\theenumi)}
\begin{enumerate}
\item $\pi^*(m)=\alpha^*(m)\beta^*(m)$, 
\item $\alpha(m)\,|\,\alpha^*(m)$,
\item $\beta^*(m)\,|\,\beta(m)$.
\end{enumerate}}

The literature of Pisano periods and restricted periods is very fascinating and extensive.  The case of second order recurrences was intensively studied by Lucas \cite{lucas}.  Beautiful treatments for 
the Fibonacci sequence are presented in \cite{robinson, wall}. In \cite{somer},  Somer presents a thorough treatment for the second order Lucas sequence of the first kind.  In a short note \cite{robinson2}, 
Robinson considers the Pisano period modulo $m$ for higher order Lucas sequences.  Finally, in \cite{vince}, Vince considers Pisano periods in the more general setting of a number field $K$, its ring of 
integers $A$, and $\mathfrak{a}$ an ideal of $A$.

\begin{remark}
Let $\pi(m)$ represents the Pisano period of the Fibonacci numbers modulo $m$.  In \cite{wall}, Wall posed the question of whether there is a prime $p$ for which $\pi(p^2) = \pi(p)$. This is still an open 
problem known as {\it Wall's question}.  A prime $p$ that satisfies 
$$F_{p-\left(\frac{p}{5}\right)}\equiv 0 \mod p^2,$$
where $F_n$ is the $n$-th Fibonacci number is known as a Wall-Sun-Sun prime. The existence of a Wall-Sun-Sun prime provides an affirmative answer to Wall's question.  In \cite{sun-sun}, Z. H. Sun and 
Z. W. Sun proved that if the first case of Fermat's last theorem was false for a prime $p$, then $p$ must be a Wall-Sun-Sun prime.  It is conjectured that there are infinitely many Wall-Sun-Sun primes.  
As of today, none has been found.  In 2007, R. McIntosh and E. Roettger \cite{mac} showed that if a Wall-Sun-Sun prime $p$ exists, then $p>2\times 10^{14}$.  In 2014, PrimeGrid showed that such a prime 
must satisfy $p>2.8 \times 10^{16}$.
\end{remark}

This manuscript considers the Pisano period of $\{S(\sigma_{n,[k_1\cdots,k_s]})\}$ modulo $p$ where $p$ is an odd prime.    Thus, from now on, let $p$  be a fixed prime such that $p$ does not divide 
$\det(A)$ where $A$ is the companion matrix of some recurrent integer sequence $\{x_n\}$.  The assumption on $p$ implies that the sequence $\{x_n \mod p\}$ is periodic instead of just eventually periodic.  

The Pisano period $\pi(p)$ and weak Pisano period $\pi^*(p)$ are closely related to the minimal polynomial of $A$ over $\mathbb{F}_p$.  In particular, A. Vince \cite{vince} proved the following result. 
\begin{theorem}
\label{weakperiodthm}
Let $\{x_n\}$ be an integer sequence satisfying a homogenous linear recurrence with integer coefficients.  Let $A$ be the companion matrix of $\{x_n\}$, $\mu_A(t)$ be the minimal polynomial of $A$ over 
$\mathbb{F}_p$ and 
\begin{equation}
\mu_A(t)=f_1(t)^{e_1}f_2(t)^{e_2}\cdots f_r(t)^{e_r}
\end{equation} 
be its irreducible factorization.  Let $\alpha_i$ be a root of $f_i(t)$ in some algebraically closed extension of $\mathbb{F}_p$ and $s_i=\lceil \log_p(e_i)\rceil$.  Then,
\begin{equation}
\pi^*(p)=\lcm(p^{s_1}\ord(\alpha_1),\cdots,p^{s_r}\ord(\alpha_r)).
\end{equation}
\end{theorem}
\noindent
The weak Pisano period will be re-labeled as $\pi^*(p; A)$ in order to make the dependence presented in Theorem \ref{weakperiodthm} more explicit.  As expected, the Pisano period $\pi(p)$, which 
already depends on $A$ (see  Theorem \ref{weakperiodthm}), also depends on the initial valued vector.  Because of this, the Pisano period will be re-labeled as $\pi(p;A;X_0)$.  It is known 
\cite{vince} that if the initial valued vector $X_0$ is a cyclic vector for $A$ over $\mathbb{F}_p$, then $\pi(p;A;X_0)=\pi^*(p;A)$.  However, if $X_0$ is not a cyclic vector, then it may 
happen that $\pi(p;A;X_0)<\pi^*(p;A)$. 

\begin{example}
Consider the sequence defined by
\begin{eqnarray}
x_n &=& 4 x_{n - 1} - 6 x_{n - 2} + 4 x_{n - 3} - 2 x_{n - 4}, \,\,\, \text{ for }n\geq 5,
\end{eqnarray}
with initial valued vector $X_0=(0,2,0,1)^T$. Let $A$ be its companion matrix.  Note that $\det(A)=2$, thus the primes 3 and 5 do not divide it. 

Consider the sequence $\{x_n \mod 3\}$.  The minimal polynomial of $A$ over $\mathbb{F}_3$ is $$\mu_{A,3}(t)=\left(t^2+1\right) \left(t^2+2 t+2\right).$$  However, note that the annihilator of 
$X_0$ is $\rho_{X_0}(t)=t^2+1$.  This is because $\rho_{X_0}(t)$ is irreducible over $\mathbb{F}_3$ and $\rho_{X_0}(A)X_0 = 0$.  Therefore, $X_0$ is not a cyclic vector for $A$ over $\mathbb{F}_3$.  
In this case, the Pisano period is $\pi(3;A;X_0)=4$, while the weak Pisano period is $\pi^*(3;A)=8$.

Consider now the sequence $\{x_n\mod 5\}$.  The minimal polynomial of $A$ over $\mathbb{F}_5$ is $$\mu_{A,5}(t)=\left(t^2+3 t+3\right) \left(t^2+3 t+4\right).$$  The vector $X_0$ is a cyclic vector 
of $A$ over $\mathbb{F}_5$ and thus, the Pisano period and weak Pisano period coincide.  In this case, $\pi(5;A;X_0)=\pi^*(5;A)=24$.
\end{example}

Theorem \ref{weakperiodthm} is one of the main tools used in this work.  In the next section, the periodicity of $\{S(\sigma_{n,[k_1,\cdots,k_s]})\mod p\}$ is considered.  In particular, 
Theorem \ref{weakperiodthm} is used to obtain upper bounds for the Pisano periods.   Later, in section \ref{avoidprimes}, these bounds are used to prove that some of the sequences 
$\{S(\sigma_{n,[k_1,\cdots,k_s]})\mod p\}$ are never zero, which implies that the corresponding symmetric functions are never balanced.

%%%%%%%%%%%%%%%%%%%%%%%%%%%%%%%%%%%%%%%%%%%%%%%%%%%%%%%%%%%%%%%%%%%%%%%%%
% The Pisano period of Exponential Sums of Symmetric Boolean functions
%%%%%%%%%%%%%%%%%%%%%%%%%%%%%%%%%%%%%%%%%%%%%%%%%%%%%%%%%%%%%%%%%%%%%%%%%

\section{Bounds and relations on the Pisano periods}
\label{boundsec}
Let $A(k_1,\cdots,k_s)$ and $\chi_{k_1,\cdots,k_s}(t)$ be the companion matrix and characteristic polynomial (resp.) associated to the minimal linear recurrence that $\{S(\sigma_{n,[k_1,\cdots,k_s]})\}$ 
satisfies.  It is known (see \cite{cm1}) that $\chi_{k_1,\cdots, k_s}(t)$ is a product (no repetitions) of terms among the list 
$$t-2, \Phi_4(t-1), \Phi_8(t-1), \cdots, \Phi_{2^r}(t-1).$$
Moreover, the term $\Phi_{2^r}(t-1)$ is always a factor of $\chi_{k_1,\cdots, k_s}(t)$.  This implies that $\det(A(k_1,\cdots,k_s))$ is a power of 2.  Therefore, if $p\neq 2$ is prime, then 
$$\{S(\sigma_{n,[k_1,\cdots,k_s]})\mod p\}$$ is periodic.  The Pisano period and weak Pisano period of this sequence are denoted by $\pi_{k_1,\cdots,k_s}(p)$ and $\pi^*_{k_1,\cdots,k_s}(p)$ respectively.

Consider the characteristic polynomial $\chi_{k_1,\cdots,k_s}(t)$ over $\mathbb{F}_p$.  It turns out that it coincides with the minimal polynomial of $A$ over $\mathbb{F}_p$.  To show this, first observe 
that $\gcd(t - 2, \Phi_{2^i} (t- 1)) = 1$.  Now, it is a well-established result that $\gcd(\Phi_m(t),\Phi_n(t)) = 1$ in $\mathbb{F}_p[t]$ for $m < n$ and neither $m$ nor $n$ is divisible by $p$. This 
implies that, for $i\neq j$, $\gcd(\Phi_{2^i} (t - 1), \Phi_{2^j} (t - 1)) = 1$ over $\mathbb{F}_p$. Thus, the claim holds because $\Phi_{2^i} (t)$ does not have repeated factors over $\mathbb{F}_p$.

This information and Theorem \ref{weakperiodthm} are all that is needed to have the following upper bound on the period modulo $p$.

\begin{theorem}
\label{mainboundthm}
Let $p>2$ be prime and let $1\leq k_1<\cdots < k_s$ be integers with $k_s>1$.  Let $r=\lfloor \log_2(k_s) \rfloor +1$.  Then,
\begin{equation}
\pi^*_{k_1,\cdots,k_s}(p)\,|\, p^{\ord_{2^r}(p)}-1,
\end{equation}
where $\ord_n(m)$ represent the order of $m$ modulo $n$. 
\end{theorem}

\begin{proof}
Let $K_i/\mathbb{F}_p$ be the splitting field of $\Phi_{2^{i+1}}(t-1)$ and $K/\mathbb{F}_p$ the splitting field of $\chi_{k_1,\cdots, k_s}(t)$.  Then,
{ \renewcommand\labelenumi{(\theenumi)}
\begin{enumerate}
\item $K_i$ is the $2^{i+1}$ cyclotomic extension over $\mathbb{F}_p$, 
\item $[K_i, \mathbb{F}_p] = \text{ord}_{2^{i+1}} (p)$,
\item  $K_i \subseteq K_{i+1}$ for all $i$.
\end{enumerate}}
Therefore, $K = K_{r-1}$. Write $\chi_{k_1,\cdots, k_s}(t) = (t-2)(t-\alpha_1)\cdots(t-\alpha_w)$ over $K$. Theorem \ref{weakperiodthm} shows that 
$\pi^*_{k_1 ,\cdots ,k_s} (p)= \lcm(\ord(2), \ord(\alpha_1 ), \cdots , \ord(\alpha_w ))$. It is clear that $\lcm(\ord(2), \ord(\alpha_1), \cdots , \ord(\alpha_w))$ divides 
$|K^{\times}| = p^{\ord_{2^r} (p)} -1$. This concludes the proof.
\end{proof}

\begin{corollary}
\label{mainboundcoro}
Let $p\equiv 1  \mod 4$ be prime and let $1\leq k_1<\cdots < k_s$ be integers with $k_s>1$.  Suppose that $p=2^r b+1$ with $b$ odd and that $k_s<2^r$.  Then,
\begin{equation}
\pi^*_{k_1,\cdots,k_s}(p)\,|\, p-1.
\end{equation}
\end{corollary}

\begin{proof}
Note that $\mathbb{F}_p$ is the splitting field of the characteristic polynomial. This concludes the proof.
\end{proof}

\begin{corollary}
\label{boundcorop3}
Let $p\equiv 3 \mod 4$ be prime and let $1\leq k_1<\cdots < k_s$ be integers with $k_s>1$.  Suppose that $p^2-1=2^r b$ with $b$ odd and that $k_s<2^r$.  Then,
\begin{equation}
\pi^*_{k_1,\cdots,k_s}(p)\,|\, p^2-1.
\end{equation}
\end{corollary}

\begin{proof}
Let $r_0=\lfloor \log_2(k_s) \rfloor$.  Observe that the hypotheses imply $r_0+1\leq r$.  Clearly, $\ord_2(p)=1$ and $\ord_{2^2}(p)=\ord_{2^3}(p)= \cdots =\ord_{2^{r_0+1}}(p)=2$.   The result follows 
from Theorem \ref{mainboundthm}.
\end{proof}

Theorem \ref{mainboundthm} and its corollaries provide upper bounds for the Pisano period of these sequences. The next result provides a lower bound for the weak Pisano period. The proof depends on 
the following concept. Consider a recurrent sequence $\{x_n\}$ which is periodic modulo $m$ for some positive integer $m$.  Fix an integer $i\geq 0$.  The {\it local period modulo $m$ at position $i$} 
is defined as the least positive integer $n$ such that $x_i \equiv x_{i+kn} \mod m$ for every integer $k$.  The local period at position $i$ is denoted by $\lambda(m;i)$.  Note that its existence follows 
from the fact that $\{x_n \mod m\}$ is periodic.  Moreover, it is clear that $\lambda(m;i)$ divides the Pisano period $\pi(m)$ of the sequence.

\begin{theorem}
\label{lowerbound}
Let $p>2$ be prime and let $1\leq k_1<\cdots < k_s$ be integers.   If $r=\lfloor\log_2(k_s)\rfloor+1$, then $$\pi^*_{k_1,\cdots,k_s}(p) \equiv 0 \mod 2^r.$$
\end{theorem}

\begin{proof}
Let $p$ be an odd prime. Observe that if the result holds for powers of two, then it is true in general.  This is because $\Phi_{2^r}(t-1)$ is always a factor of $\chi_{k_1,\cdots,k_s}(t)$ and therefore 
$\pi^*_{2^{r-1}}(p)\,|\, \pi^*_{k_1,\cdots,k_s}(p)$.

Suppose that $k$ is a power of two, say $k=2^{l}$ for $l\geq 1$.  Recall that
\begin{equation}
S(\sigma_{n,2^l}) = \sum_{j=0}^n (-1)^{\binom{j}{2^l}}\binom{n}{j}
\end{equation} 
and that the characteristic polynomial of the minimal recurrence associated to it is $\Phi_{2^{l+1}}(t-1)$.  The degree of this polynomial is $2^{l}$.  Thus, in order to construct  
$\{S(\sigma_{n,2^{l}})\}$ from this recursion, one needs $2^{l}$ initial values.  

The first $2^l$ values of $S(\sigma_{n,2^{l}})$ (starting from $n=1$) are
$$2,4,8,\cdots, 2^{2^l-1} , 2^{2^l}-2.$$
These values can be used to obtain every other value of $S(\sigma_{n,2^l})$.  However, note that in this case, one can start the recurrence at $n=0$ instead of $n=1$.  This is because the value at 
$n=0$ is 1 and $-\Phi_{2^{l+1}}(t-1)+t^{2^l}$ returns $2^{2^l}-2$ when $t=2$.  Thus, the values 
$$1, 2,4,8,\cdots, 2^{2^l-1}$$
can be used to generate the value $2^{2^l}-2$ and therefore any other value of $S(\sigma_{n,2^l})$.  Moreover, since 
$$(-\Phi_{2^{l+1}}(t-1)+t^{2^l}+2)/2$$ 
returns $2^{2^l-1}$ when $t=2$, then the sequence defined by $x_0=0$ and $x_j = 2^{j-1}$ for $j=1,2,\cdots, 2^l-1$ is such that $x_0=0$ and $x_n = S(\sigma_{n-1,2^l})$ for $n\geq 1$.  

Observe $x_{2^{l+1}D}=S(\sigma_{2^{l+1}D-1,2^l})=0=x_0$ for every natural number $D$ (these are trivially balanced cases).  This implies that $\lambda(p;0)\,|\, 2^{l+1}$.  However, 
$\lambda(p;0)\neq 2^j$ for $0\leq j<l+1$, because 
$$x_{2^j} = 2^{2^j-1} \not\equiv 0 =x_0 \mod p.$$ 
Thus, $\lambda(p;0)=2^{l+1}$ and the result follows from the fact that $$\lambda(p;0)\,|\, \pi_{2^l}(p)\,|\, \pi^*_{2^l}(p).$$
This concludes the proof.
\end{proof}

%Theorems \ref{mainboundthm} and \ref{lowerbound} imply that for $p>2$ prime and $1\leq k_1<\cdots<k_s$ integers,
%\begin{equation}
%2^r \leq \pi^{*}_{k_1,\cdots,k_s}(p) \leq p^{\ord_{2^r}(p)}-1,
%\end{equation}
%where $r=\lfloor \log_2(k_s) \rfloor+1$.    

\begin{example}
\label{niceex1}
Consider the prime $p=41$.  The terms corresponding to $k=2,3,\cdots, 15$ of the sequence of Pisano periods $\{\pi_k(41)\}_{k\geq 2}$ are given by
$$20, 20, 40, 40, 40, 40, 1680, 1680, 1680, 1680, 1680, 1680, 1680, 1680.$$
Note that
\begin{equation}
\pi_2(41)=\pi_3(41)=20\equiv 0 \mod 4
\end{equation}
as Theorem \ref{lowerbound} predicted.  Moreover, 20 is a divisor of 40, as predicted by Corollary \ref{mainboundcoro}.  From $k=4$ to $k=7$, the period is $\pi_k(41)=40$, thus the bound provided by 
Theorem \ref{mainboundthm} can be attained. Furthermore, $\pi_k(41)=40 \equiv 0 \mod 8$.  Finally, from $k=8$ to $k=15$ the period is $\pi_k(41)=1680=41^2-1$ and $1680\equiv 0 \mod 16$.
\end{example}

\begin{example}
\label{niceex2}
Consider now the prime $p=13$.  The terms corresponding to $k=2,3,\cdots, 15$ of the sequence of Pisano periods $\{\pi_k(13)\}_{k\geq 2}$ are given by
$$12, 12, 168, 168, 168, 168, 5712, 5712, 5712, 5712, 5712, 5712, 5712, 5712.$$
The reader can verify that Theorems \ref{mainboundthm} and \ref{lowerbound} hold.  Observe that from $k=8$ to $k=15$, the period is $\pi_k(13)=5712$, which is a divisor of $13^4-1=28560$.  Therefore, 
the bound provided by Theorem \ref{mainboundthm} is not always sharp.
\end{example}

Examples \ref{niceex1} and \ref{niceex2} show how Theorems \ref{mainboundthm} and \ref{lowerbound} work, but also suggest a beautiful relation between the Pisano periods of the sequences considered in 
this article.  Observe that in both cases, $\pi_{k_1}(p)$ divides $\pi_{k_2}(p)$ for $k_1<k_2$.  Moreover, in both examples one has $\pi_{k_1}(p)=\pi_{k_2}(p)=\pi_{2^{r-1}}(p)$ when 
$2^{r-1}\leq k_1,k_2 <2^{r}$.  The next series of results shows that this relation holds for weak Pisano periods.

\begin{lemma}
\label{rootrelation}
Let $p>2$ be prime and $r>0$ be an integer.  Let $\alpha_1,\cdots,\alpha_{2^r}$ be the roots of $\chi_{2^r}(t)$ in some extension of $\mathbb{F}_p$.  Similarly, let $\beta_1,\cdots,\beta_{2^{r+1}}$ be 
the roots of $\chi_{2^{r+1}}(t)$ in some extension of $\mathbb{F}_p$.  Then,
$$\lcm(\ord(\alpha_1),\cdots,\ord(\alpha_{2^{r}}))\,|\, \lcm(\ord(\beta_1),\cdots,\ord(\beta_{2^{r+1}})).$$
\end{lemma}

\begin{proof}
Let $\alpha$ be a root of $\chi_{2^r}(t)=\Phi_{2^{r+1}}(t-1)$ in some extension of $\mathbb{F}_p$.  This is, $(\alpha-1)^{2^r}+1=0$.  Let $\beta$ be a root of $t^2-2t+2-\alpha$, i.e. $(\beta-1)^2+1=\alpha$.  Note that
\begin{eqnarray*}
\chi_{2^{r+1}}(\beta) &=& (\beta-1)^{2^{r+1}}+1\\
&=& ((\beta-1)^2)^{2^{r}}+1\\
&=& ([(\beta-1)^2+1]-1)^{2^{r}}+1\\
&=& (\alpha-1)^{2^{r}}+1\\
&=& 0,
\end{eqnarray*}
and therefore $\beta$ is a root of $\chi_{2^{r+1}}(t)$.  The converse is also true, i.e. if $\beta$ is a root of $\chi_{2^{r+1}}(t)$ in some extension of $\mathbb{F}_p$, then $\alpha=(\beta-1)^2+1$ is such 
that $\chi_{2^r}(\alpha)=0$.  Since $t^2-2t+2-\alpha$ does not have repeated roots, then all roots of $\chi_{2^{r+1}}(t)$ can be obtained from $\alpha_1,\cdots,\alpha_{2^r}$ by finding the roots of 
$t^2-2t+2-\alpha_i$ for $i=1,2,3,\cdots, 2^r$. 

The proof for $r=1$ (instead of the general case) is presented next.  This is done for the simplicity of the writing.  The general case can be done using the same argument and induction.  Let $\alpha_1$ 
and $\alpha_2$ be the roots of $\chi_2(t)$.  Write $t^2-2t+2 = (t-\alpha_1)(t-\alpha_2)$ and observe that
\begin{eqnarray*}
\alpha_1+\alpha_2 &=& 2\\
\alpha_1\alpha_2 &=& 2.
\end{eqnarray*}
Let $\beta_1,\beta_2,\beta_3,\beta_4$ be the roots of $\chi_4(t)$.  Say $\beta_1,\beta_2$ are the roots of $t^2-2t+2-\alpha_1$ and $\beta_3,\beta_4$ are the roots of $t^2-2t+2-\alpha_2$.  Then,
\begin{eqnarray*}
\beta_1+\beta_2 &=& 2\\
\beta_1\beta_2 &=& 2-\alpha_1 = \alpha_2\\
\beta_3+\beta_4 &=& 2\\
\beta_3\beta_4 &=& 2-\alpha_2 = \alpha_1.
\end{eqnarray*}
 Let $m_2 = \lcm(\ord(\alpha_1),\ord(\alpha_2))$ and $m_4 = \lcm(\ord(\beta_1),\ord(\beta_2),\ord(\beta_3),\ord(\beta_4)).$  Observe that
$$\alpha_1^{m_4} = (\beta_3\beta_4)^{m_4} = \beta_3^{m_4}\beta_4^{m_4} = 1\cdot 1 = 1.$$
Similarly, $\alpha_2^{m_4}=1$.  It follows that $m_2\,|\,m_4$.  This concludes the proof.  
\end{proof}

\begin{theorem}
\label{periodreduction}
Let $p>2$ be prime and $1\leq 1<k_1<\cdots<k_s$ be integers with $k_s>1$.  Let $r=\lfloor\log_2(k_s)\rfloor+1$.  Then,
$$\pi^{*}_{k_1,\cdots,k_s}(p) = \pi^{*}_{2^{r-1}}(p).$$
\end{theorem}

\begin{proof}
Recall that $\chi_{k_1,\cdots, k_s}(t)$ is a square-free product of terms among the list
$$t-2, \Phi_4(t-1), \Phi_8(t-1), \cdots, \Phi_{2^r}(t-1),$$
and that $\Phi_{2^r}(t-1)=\chi_{2^{r-1}}(t)$ is always a factor of $\chi_{k_1,\cdots, k_s}(t)$.  Let $$\alpha_1,\cdots, \alpha_{2^{r-1}}, \alpha_{2^{r-1}+1},\cdots, \alpha_w$$ be the roots of 
$\chi_{k_1,\cdots, k_s}(t)$ in some extension of $\mathbb{F}_p$.  Suppose that $\alpha_1,\cdots, \alpha_{2^{r-1}}$ are the roots that come from the factor $\Phi_{2^r}(t-1)$.  Theorem \ref{weakperiodthm} 
and Lemma \ref{rootrelation} imply
\begin{eqnarray*}
\pi^{*}_{k_1,\cdots,k_s}(p) &=& \lcm(\ord(\alpha_1),\cdots, \ord(\alpha_{2^{r-1}}), \ord(\alpha_{2^{r-1}+1}),\cdots, \ord(\alpha_w))\\
&=& \lcm(\ord(\alpha_1),\cdots, \ord(\alpha_{2^{r-1}}))\\
&=& \pi^{*}_{2^{r-1}}(p).
\end{eqnarray*}
This concludes the proof.
\end{proof}

\begin{corollary}
Let $p>2$ be prime. If $2\leq k_1 < k_2$ are integers, then $\pi^{*}_{k_1}(p)$ divides $\pi^{*}_{k_2}(p)$.  Moreover, if $2^{r-1} \leq k_1, k_2 < 2^r$, then 
$\pi^{*}_{k_1}(p)=\pi^{*}_{k_2}(p)=\pi^{*}_{2^{r-1}}(p).$
\end{corollary}

\begin{proof}
This is a direct consequence of Lemma \ref{rootrelation} and Theorem \ref{periodreduction}.
\end{proof}

In the next section Pisano periods are used to show that some symmetric Boolean functions are not balanced.  The idea is simple: given a Boolean function $\sigma_{n,[k_1,\cdots,k_s]}$ and a prime $p$, 
check that $S(\sigma_{n,[k_1,\cdots,k_s]})\not\equiv 0 \mod p$ for $n=1,2,\cdots, \pi_{k_1,\cdots,k_s}(p)$.  Note that in practice, it suffices to use weak Pisano periods instead of Pisano periods.  
The reason for this is that  there are at most a finite amount of primes $p$ for which the initial valued vector $X_0$ of $\{S(\sigma_{n,[k_1,\cdots,k_s}])\}$ is not a cyclic vector for the companion 
matrix $A(k_1,\cdots, k_s)$.  Thus, for almost all primes one is guaranteed to have $\pi_{k_1,\cdots,k_s}(p)=\pi^*_{k_1,\cdots,k_s}(p)$.  In fact, for powers of two, one has $\pi_{2^l}(p)=\pi^*_{2^l}(p)$ 
for every odd prime.   This is because for $X_0=(1,2,4,8,\cdots,2^{2^l-1})^T$ and $A=A(2^l)$, the determinant of the matrix  
\begin{equation}
B=(X_0, AX_0, A^2X_0,\cdots, A^{2^{2^l-1}-1}X_0)
\end{equation}  
is $\det(B)=-2$ if $l=1$ and $\det(B)=2^{2^l-1}$ if $l\neq 1$.  Therefore, the vector $X_0$ is a cyclic vector of $A$ over $\mathbb{F}_p$ for every odd prime and the claim holds.

%Observe that
%\begin{equation}
%B \sim \left(
%\begin{array}{rrrrrrr}
% 1 & 2 & 4 & \cdots & 2^{2^l-3} & 2^{2^l-2} & 2^{2^l-1} \\
% 0 & 0 & 0 & \cdots & 0 & 0 & -2 \\
% 0 & 0 & 0 & \cdots & 0 & -2 & * \\
% 0 & 0 & 0 & \cdots & -2 & * & * \\
% \vdots & \vdots & \vdots & \vdots & \vdots & \vdots & \vdots \\
% 0 & 0 & -2 & \cdots & * & * & * \\
% 0 & -2 & * & \cdots & * & * & * \\
%\end{array}
%\right),
%\end{equation}
%which implies that 
%\begin{equation}
%\det(B)= \begin{cases}
% -2 & \text{if } l=1 \\
% 2^{2^l-1}& \text{if } l\neq 1.
%\end{cases}
%\end{equation}

%%%%%%%%%%%%%%%%%%%%%%%%%%%%%%%%%%%%%%%%%%%%%%
% Section: Avoiding primes
%%%%%%%%%%%%%%%%%%%%%%%%%%%%%%%%%%%%%%%%%%%%%%
\section{Avoiding primes}
\label{avoidprimes}

Suppose that one encounters the problem of showing that a sequence $\{x_n\}$ is never zero, i.e. $x_n\neq 0$ for every $n$.  In practice, this can be a very difficult problem to tackle.  However, if one 
can show that there is a prime $p$ such that it does not divide $x_n$ for every $n$, then the sequence is never zero.  An integer sequence $\{x_n\}$ is said to {\it avoid} a prime $p$ if $p$ does not 
divide $x_n$ for every $n$, otherwise $\{x_n\}$ {\it cannot avoid}  $p$ (it is common in the literature to say that, in the later case, $p$ divides $\{x_n\}$).   Of course, there are sequences that cannot 
avoid any prime, yet they are never zero.  The most simple example is $\{n\}_{n\in \mathbb{N}}$, a sequence that is even periodic modulo $m$ for every positive integer $m$.

Consider the sequence given by the exponential sum of an elementary symmetric polynomial $\{S(\sigma_{n,k})\}$, which is at the root of Cusick-Li-St$\check{\mbox{a}}$nic$\check{\mbox{a}}$'s conjecture.  
It is known that if $k$ is odd, then $S(\sigma_{n,k})$ is never zero \cite{cusick1}.  On the other hand, if $k=2^r$ is a power of two, then $n=2^{r+1}m-1$, for $m$ a positive integer, is such that 
$S(\sigma_{n,k})=0$ (these are trivially balanced cases).  Therefore, the first case that neither of these two results rule out is $\{S(\sigma_{n,6})\}$.  

The sequence $\{S(\sigma_{n,6})\}$ satisfies the linear recurrence
\begin{equation}
x_n  = 8 x_{n-1} -28 x_{n-2}+ 56 x_{n-3} -70x_{n-4}+ 56x_{n-5} -28x_{n-6}+ 8x_{n-7}.
\end{equation}
Choose the prime $p=3$, which is the smallest one for which $\{S(\sigma_{n,6}) \mod p\}$ is periodic.
In this case, the Pisano period is given by $\pi_6(3)=8$.  The fundamental period for this sequence is 
$$2, 1, 2, 1, 2, 2, 1, 1.$$
This implies that $3$ does not divide $S(\sigma_{n,6})$ for any $n$, i.e. the sequence avoids the prime 3.  Thus the Boolean function $\sigma_{n,6}$ is not balanced for any $n$.  

The fact that $\sigma_{n,6}$ is never balanced is already known and it follows from \cite[Th. 4, p. 2805]{CV05}, however, the above discussion shows how the technique works.
The same technique can be used to prove that other symmetric polynomials are not balanced because they avoid primes.   However, the choice of the prime $p=3$ seems to be a rather bad one.  In fact, it looks 
like the only sequence of the form $\{S(\sigma_{n,k})\}$ that avoids 3 is the one when $k=6$ and there might be a heuristic argument behind this.   

In the case of $\{S(\sigma_{n,k})\mod 3\}$, for $2\leq k \leq 7$, the Pisano period is 8.  The probability that a vector of length 8 with entries in $\mathbb{F}_3$, which is chosen at random, does not contain 0 is 
$$\left(\frac{2}{3}\right)^8\approx 0.03901844231.$$
The situation gets worse for $k>7$.  In particular, for $8\leq k \leq 15$, the Pisano period is given by $\pi_k(3)=80$ and $(2/3)^{80}$ is already minuscule.  Thus, one expects that if a vector in $\mathbb{F}_3^{80}$ is chosen at random, then it contains a good amount of zeros (the expected number of zeros is $80/3=26.66\bar{6}$).  Indeed, let 
\begin{equation}
z_k(p)=\text{amount of }0\text{'s in } \{S(\sigma_{n,k}) \mod p\,|\, 1\leq n \leq \pi_k(p)\}.
\end{equation}
Then,
$$\begin{array}{|c|c|c|c|c|c|c|c|c|}
\hline
k & 8 & 9 & 10 & 11 & 12 & 13 & 14 & 15 \\
z_k(3)&  38 & 21 & 27 & 26 & 17 & 22 & 17 & 21 \\
\hline
\end{array}$$
For $16\leq k \leq 31$, the Pisano period is $\pi_k(3)=6560$ and one has
$$
\begin{array}{|c|c|c|c|c|c|c|c|c|}
\hline
k& 16 & 17 & 18 & 19 & 20 & 21 & 22 & 23 \\
z_k(3)& 2402 & 2079 & 2100 & 2117 & 2081 & 2143 & 2081 & 2133 \\
 \hline
k& 24 & 25 & 26 & 27 & 28 & 29 & 30 & 31 \\
z_k(3)& 2091 & 2324 & 2204 & 2169 & 2194 & 2108 & 2049 & 2153 \\
\hline
\end{array}$$
For $32\leq k\leq 63$, the period is $\pi_k(3)=3^{16}-1=43046720$ and for $64\leq k\leq 127$, the value is $\pi_k(3)=3^{32}-1=1853020188851840$.  For $128\leq k< 255$, $\pi_k(3)=3^{64}-1$ and for $256\leq k <511$, the period is $\pi_k(3)=3^{128}-1$, which is a number with 62 digits.  This data suggests that it should not be ``expected" to find an elementary symmetric polynomial of degree $k$ bigger than 7 such that $\{S(\sigma_{n,k})\}$ avoids 3.

The above analysis does not suggest the absence of symmetric Boolean functions for which their exponential sums avoid the prime 3, it only suggests that for elementary symmetric Boolean functions.   However, given the magnitude of the Pisano periods $\pi_k(3)$,  if there are other symmetric Boolean functions for which their exponential sums avoid 3, then it is probable that their degrees are less than or equal to 7.   In fact, given that there are $127$ symmetric Boolean functions of degree less than or equal to 7, then it is almost certain that they exist.  Indeed, an exhaustive search found all the symmetric Boolean functions of degree less than or equal to 7 such that their exponential sums avoid the prime 3,
\begin{equation}
\label{avoid3}
\begin{array}{lllll}
\sigma_{n,[2,3,5]}& \sigma_{n,6} & \sigma_{n,[4,6]} & \sigma_{n,[4,5,6]} & \sigma_{n,[2,3,5,6]}\\
\sigma_{n,[4,7]} & \sigma_{n,[4,5,6,7]} & \sigma_{n,[2,3,4,5,7]} & \sigma_{n,[2,3,4,5,6,7]}. & \\
\end{array}
\end{equation}
These might be all of the symmetric Boolean functions with this property.

\begin{remark}
The symmetric Boolean function $\sigma_{n,[2,3,5]}$ is peculiar in the sense that every other symmetric Boolean function in (\ref{avoid3}) is asymptotically not balanced while $\sigma_{n,[2,3,5]}$ is not.  This means that even without the knowledge that their exponential sums avoid the prime 3, one knows that every symmetric Boolean function in (\ref{avoid3}) different from $\sigma_{n,[2,3,5]}$ is not balanced  for $n$ big enough.  The same cannot be said about $\sigma_{n,[2,3,5]}$, thus the fact that its exponential sum avoids 3 holds the key to show that it is not balanced.  Figure \ref{figsigma532} is a graphical representation of $S(\sigma_{n,[2,3,5]})/2^n$.  Observe that even though the graph oscillates between positive and negative numbers and can get as close to zero as one desires, it is never zero because $\{S(\sigma_{n,[2,3,5]})\}_{n\in\mathbb{N}}$ avoids the prime 3.
\begin{figure}[htb]
\caption{Graphical representation of $S(\sigma_{n,[2,3,5]})/2^n$.}
\centering
\includegraphics[width=2.8in]{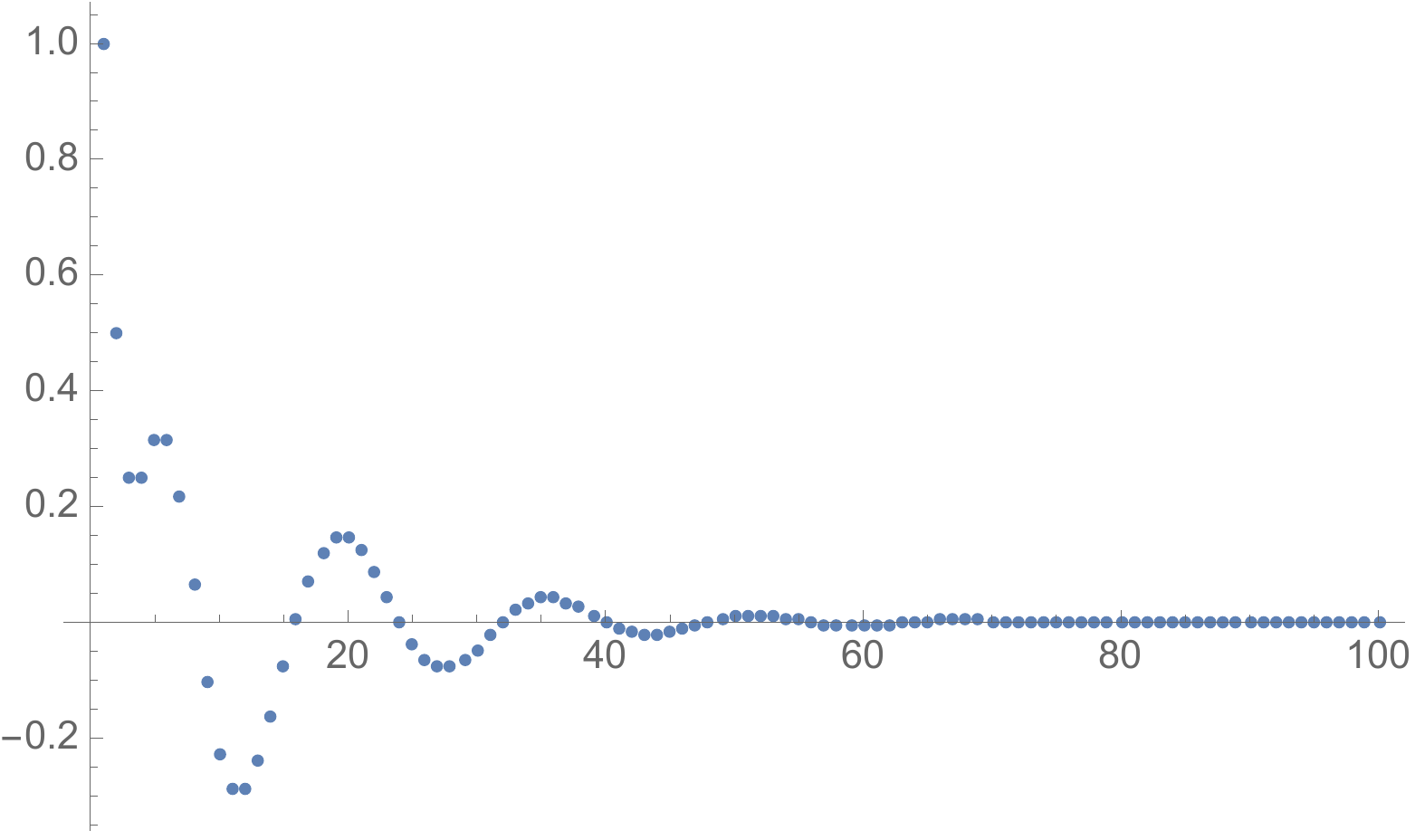}
\label{figsigma532}
\end{figure}
\end{remark}

The quest now is to identify primes that are excellent candidates for these sequences to avoid.  These should be primes $p$ with ``small" period relative to $p$.  Theorem \ref{mainboundthm} and Corollary \ref{mainboundcoro} offer a clue.   In fact, the case $p\equiv 1 \mod 4$ provides some optimism. Recall that  Corollary \ref{mainboundcoro} states that if $p=2^rb+1$ with $b$ odd, then, for $1\leq k_1<\cdots<k_s< 2^r$, the Pisano period $\pi_{k_1,\cdots,k_s}(p)$ divides $p-1$ and
\begin{equation}
0.3678794412\approx \frac{1}{e} < \left(\frac{p-1}{p}\right)^{p-1} \leq \left(\frac{4}{5}\right)^4=0.4096.
\end{equation} 
Therefore, heuristically, if a symmetric Boolean function is not balanced and its degree is $k_s$, then there is a good chance that its exponential sum avoids a prime of the form $2^r b+1$ where $b$ is an integer and $r=\lfloor\log_2(k_s)\rfloor+1$.  

The case $p\equiv 3 \mod 4$ is different. Theorem \ref{mainboundthm} implies that the smallest possible Pisano period is a divisor of $p^2-1$ and 
\begin{equation}
\left(\frac{p-1}{p}\right)^{p^2-1}\to 0 \,\,\text{ as }p\to \infty.
\end{equation}
Thus, if the exponential sum of an individual symmetric Boolean function is considered, then these primes do not seem to be good candidates for the sequence to avoid.  However, as it was the case of the prime 3, there is strength in numbers.  To be precise, if $p\equiv 3 \mod 4$ and $p^2-1$ is highly divisible by 2, then Corollary \ref{boundcorop3} implies that the exponential sum of every symmetric Boolean function of degree less than the higher power of 2 that divides $p^2-1$ has Pisano period divisor of $p^2-1$.  If the number of these symmetric Boolean functions is big enough, one may have that the expected number of them that avoid $p$ is bigger than 1 and therefore these primes also become good candidates.  For example, consider the prime $p=7$.  In this case, $p^2-1=48$ and $48=16\cdot 3$.  There are $2^{15}-1=32767$ symmetric Boolean functions of degree less than 16 and 
\begin{equation}
(2^{15}-1)\left(\frac{6}{7}\right)^{48} \approx 20.0443.
\end{equation}
This suggests that about 20 of them have exponential sums that avoid 7.  An exhaustive search found out that there are exactly 18 of them.  These are
\begin{equation}
\begin{array}{lll}
\sigma_{n,[2, 3, 4, 5, 9, 10, 11]} & \sigma_{n,[2, 3, 5, 7, 9, 12]} & \sigma_{n,[3, 4, 5, 7, 10, 11, 12]}\\
\sigma_{n,[3, 6, 8, 9, 12, 13]} & \sigma_{n,[2, 3, 6, 7, 9, 12, 13]} & \sigma_{n,[2, 3, 5, 7, 8, 9, 10, 11, 12, 13]}\\
\sigma_{n,[2, 3, 4, 5, 6, 7, 9, 10, 11, 12, 13]} & \sigma_{n,[3, 6, 10, 13, 14]} & \sigma_{n,[2, 3, 5, 7, 9, 14]}\\
\sigma_{n,[5, 6, 8, 10, 13, 14]} & \sigma_{n,[8, 9, 10, 11, 12, 14]} & \sigma_{n,[8, 9, 10, 12, 13, 14]}\\
\sigma_{n,[3, 4, 5, 6, 8, 9, 14]} & \sigma_{n,[3, 5, 8, 9, 10, 12, 14]} & \sigma_{n,[3, 5, 8, 10, 11, 12, 13, 14]}\\
\sigma_{n,[5, 6, 8, 11, 13, 15]} & \sigma_{n,[3, 4, 5, 7, 8, 9, 15]} & \sigma_{n,[2, 3, 8, 9, 10, 12, 14, 15]}.
\end{array}
\end{equation}

\begin{remark}
The first prime that it is not a ``good candidate" in the sense described above is the prime 11.  From degree 4 to degree 7, the weak Pisano period is 120.  From degree 8 to degree 15, the weak Pisano period is 14640 and from degree 16 to degree 31 the weak Pisano period is already an astonishing number: 214358880. It has been verified that the prime 11 divides the sequence of exponential sums of every symmetric Boolean function degree less than or equal to 24 and there is a total of $2^{24}-1 = 16777217$ of such functions.  The prime 19 is similar.  Moreover, it appears that for every prime $p>2$, there is a number $\mathfrak{K}(p)$ such that $p$ divides the sequence of exponential sums of every symmetric Boolean function of degree bigger than or equal to $\mathfrak{K}(p)$.  For example, computer experiments suggest that $\mathfrak{K}(3)=8$, $\mathfrak{K}(7)=16$ and $\mathfrak{K}(11)=1.$
\end{remark}

Now take a step back and consider once again the case of an elementary symmetric Boolean function.  Since the focus is on an individual function, then only primes $p\equiv 1\mod 4$ are considered, as they seem to be excellent candidates to avoid.  Suppose that $k$ is a positive integer which is not a power of two.  Define $r$ by $2^{r-1}<k<2^r$ and let $\mathfrak{p}(k)$ be the smallest prime $p$ of the form $p=2^rm+1$, $m$ a positive integer, such that $\{S(\sigma_{n,k})\}$ avoids $p$.  For example, $\mathfrak{p}(3)=5$.  Of course, in general this number might not exist, however its existence would imply Cusick-Li-St$\check{\mbox{a}}$nic$\check{\mbox{a}}$'s conjecture.  It has been verified that the prime $\mathfrak{p}(k)$ exists for all $k$ in the range $1< k < 2048$ that are not powers of 2 showing that $\sigma_{n,k}$ is not balanced for any $n$ and such $k$'s.  For for example, $\mathfrak{p}(1292)=176129$.   The reader is invited to see the table included in Appendix \ref{app:table}, which contains $\mathfrak{p}(k)$ for $k\leq 263.$  Note that this supports the claim that primes $p\equiv 1 \mod 4$ are good candidates to avoid.  This information leads the authors to the following conjecture:

\begin{conjecture}
\label{theprimeexists}
Let $k$ be a natural number which is not a power of two.  The prime $\mathfrak{p}(k)$ exists.
\end{conjecture}

\begin{remark}
As mentioned before, an affirmative answer to this conjecture implies that Cusick-Li-St$\check{\mbox{a}}$nic$\check{\mbox{a}}$'s conjecture is true.  However, Conjecture \ref{theprimeexists} is much stronger.  Recurrence (\ref{mainrec}) can be used to extend $\{S(\sigma_{n,k})\}_{n\in\mathbb{N}}$ infinitely to the left (non-positive values of $n$) with the compromise that the value of $S(\sigma_{n,k})$ for such an $n$ may be a rational number instead of an integer.   For example, the sequence $\{S(\sigma_{n,6}\})\}_{n\in\mathbb{Z}}$ is given by 
$$\cdots,\frac{233}{128},\frac{65}{64},-\frac{19}{32},-\frac{19}{16},-\frac{5}{8},\frac{1}{4},1, 2, 4, 8, 16, 32, 62,112,\cdots.$$
The existence of the prime $\mathfrak{p}(k)$ implies that $\{S(\sigma_{n,k})\}_{n\in \mathbb{Z}}$ is never zero, regardless of the integer value of $n$.  This is much stronger than asserting that $S(\sigma_{n,k})\neq 0$ for $n\geq k$.  For instance, $S(\sigma_{n,[3,2,1]})\neq 0$ for $n\geq 3$, which means, in terms of the theory of Boolean functions, that $\sigma_{n,[3,2,1]}$ is not balanced.  However, every prime divides $\{S(\sigma_{n,[3,2,1]})\}$ for the simple reason that the value 0 is obtained when $n=1$.  A Boolean function $\sigma_{n,[k_1,\cdots,k_s]}$ is said to be {\em strongly non-balanced} if $S(\sigma_{n,[k_1,\cdots,k_s]})\neq 0$ for any integral value of $n$.  For example, $\sigma_{n,6}$ is strongly non-balanced, while $\sigma_{n,[3,2,1]}$ is non-balanced, but not strongly non-balanced.  Of course, if Conjecture \ref{theprimeexists} is true, it would imply that every elementary symmetric Boolean function $\sigma_{n,k}$ is strongly non-balanced for $k$ not a power of two.
\end{remark}

The above argument provides examples of symmetric Boolean functions whose exponential sums avoid primes.  This leads to the following question: given a sequence of exponential sums like $\{S(\sigma_{n,[k_1,\cdots,k_s]})\}$, how many primes are avoided by it?   In order to try to answer this question, the focus is now shifted to the set of primes avoided by a given sequence of exponential sums.  Since all the theory discussed throughout the article also applies to perturbations, then they are also included in the study below.  These perturbations will prove to be useful when the case of degree two is considered.

Let $1\leq k_1<\cdots <k_s$ and define
\begin{equation}
Av(k_1,\cdots,k_s)=\{p \text{ prime} \,\,|\,\, S(\sigma_{n,[k_1,\cdots,k_s]}) \text{  avoids }p \}
\end{equation}
Similarly, for a perturbation $\sigma_{n,[k_1,\cdots,k_s]}+F({\bf X})$, the set of primes avoided by $\{S(\sigma_{n,[k_1,\cdots,k_s]}+F({\bf X}))\}$ is denoted by $Av(k_1,\cdots,k_s; F({\bf X}))$.  Following the literature, the complements of the sets $Av(k_1,\cdots,k_s)$ and $Av(k_1,\cdots,k_s; F({\bf X}))$, i.e. the sets of primes that divide the sequences $\{S(\sigma_{n,[k_1,\cdots,k_s]})\}$ and $\{S(\sigma_{n,[k_1,\cdots,k_s]}+F({\bf X}))\}$, are denoted by $P(k_1,\cdots,k_s)$ and $P(k_1,\cdots,k_s; F({\bf X}))$, respectively. 
Suppose $\{x_n\}_{n\in \mathbb{N}}$ is an integer sequence that satisfies a homogenous linear recurrence with integer coefficients.  The set of primes dividing the sequence $\{x_n\}$ is denoted by $P(x_n)$.  The sequence $\{x_n\}$ is said to be {\it degenerate} if the associated characteristic polynomial is either inseparable or has two different roots whose quotient is a root of unity, otherwise the sequence is said to be {\it non-degenerate}.   In 1921,  P{\'o}lya  proved that if the sequence is non-degenerate, then $P(x_n)$ is infinite.

Consider the sequence $\{S(\sigma_{n,[k_1,\cdots,k_s]})\}$. It has been established that its characteristic polynomial is a square-free product of terms among the list 
$$t-2, \Phi_4(t-1), \Phi_8(t-1), \cdots, \Phi_{2^r}(t-1),$$
where $r=\lfloor \log_2(k_s) \rfloor+1$, and that $\Phi_{2^r}(t-1)$ is always a factor.  Consider the numbers $\alpha_1=1-\exp(\pi i/2^r)$
 and $\alpha_2=1+\exp((2^r-1)\pi i/2^r)$, which are two roots of $\Phi_{2^r}(t-1)$.  Observe that  quotient $\alpha_1/\alpha_2 = -\exp(\pi i/2^r)$ is a root of unity and so $\{S(\sigma_{n,[k_1,\cdots,k_s]})\}$ is degenerate (the same reason implies that sequences of exponential sums of the perturbations considered in this article are also degenerate).  This means that P{\'o}lya's result does not guarantee that $P(k_1,\cdots,k_s)$ is infinite and thus, $Av(k_1,\cdots,k_s)$ may very well be infinite or even dense within the set of primes.  However, it can also be finite, which is clearly the case if one of the initial conditions is 0.   For instance, if the degrees $1\leq k_1<\cdots<k_s$ are such that $k_1=1$, then it is clear from the definition of $S(\sigma_{n,[1,k_2,\cdots,k_s]})$ as a binomial sum that $S(\sigma_{1,[1,k_2,\cdots,k_s]})=0$, i.e. the first initial condition is 0.   This implies that $Av(1,k_2,\cdots,k_s)=\emptyset$, which is equivalent to say that every prime $p$ divides infinitely many terms of the sequence $\{S(\sigma_{n,[1,k_2,\cdots,k_s]})\}$.  The same applies to $\{S(\sigma_{n,[2,k_2,\cdots,k_s]})\}$ when $k_2\geq 4$ because the third initial condition is 0.

\begin{definition}
Suppose that $S$ is a subset of the set of primes.  The {\em natural density} of $S$  is defined by
\begin{equation}
\delta(S)=\lim_{N\to\infty} \frac{\#\{p\in S: p\leq N\}}{\#\{p \text{ prime }: p\leq N\}}
\end{equation}
whenever this limit exists.
\end{definition}
 
The fact that these sequences are degenerate can be explicitly seen for the case of the exponential sum of $\sigma_{n,2}$ and its perturbations.   In particular, for any perturbation $\sigma_{n,2}+F({\bf X})$, the set $P(2;F({\bf X}))$ is either the set of all primes (because the sequence $\{S(\sigma_{n,2}+F({\bf X}))\}_{n>j}$ has zeros in it) or a finite set.  In particular, the following result holds.

\begin{theorem}
\label{pert2}
Let $n$ and $j$ be positive integers. Suppose that $j<n$ and let $F({\bf X})$ be a binary polynomial in the variables $X_1,\cdots, X_j$ (the first $j$ variables in $X_1,\cdots, X_n$).  Let $D_F(n)=\sigma_{n,2}+F({\bf X})$.  If 
\begin{equation}
\frac{S(D_F(2+j))}{S(D_F(1+j))}-1 \notin \{0,\pm 1, \infty\},
\end{equation}
then $D_F(n)$ is not balanced for every $n>j$.  Moreover, the set of primes that divide the sequence \begin{equation}
\label{distor}
\{S(D_F(n))\}_{n>j} 
\end{equation}
is finite, which implies that the set of primes avoided by (\ref{distor}) has density 1.
\end{theorem}

\begin{proof}
Let $d(n)=S(D_F(n))$.  Then, the sequence $\{d(n)\}_{n>j}$ satisfies the recurrence
\begin{equation}
x_n = 2x_{n-1}-2x_{n-2}.
\end{equation}
Finding the roots of the characteristic polynomial and solving the corresponding system for this perturbation, one finds
\begin{eqnarray}
\nonumber
d(n)&=&2^{\frac{1}{2} (-j+n+1)-1} \left(d(j+1) \sin\left(\frac{ \pi (n-j+1)}{4}\right)\right)\\
&&+2^{\frac{1}{2} (-j+n+1)-1} \left((d(j+1)-d(j+2)) \cos\left(\frac{\pi(n-j+1)}{4}\right)\right).
\end{eqnarray}
Therefore, $D_F(n)$ is balanced, i.e. $d(n)=0$, if and only if 
\begin{equation}
\frac{d(j+2)}{d(j+1)}-1=\tan \left((n-j+1)\frac{\pi}{4}\right).
\end{equation}
Thus, if 
$$\frac{d(j+2)}{d(j+1)}-1\notin \{0,\pm 1,\infty\},$$
then the perturbation $D_F(n)$ is not balanced for any $n>j$.   

Suppose that the perturbation $D_F(n)$ is not balanced.  The quest now is to find the set of primes that divide $d(n)$.  Dividing the values of $d(n)$ according to the remainders of $n$ on division by 4, one finds that a prime $p$ divides $d(n)$ if and only if it divides one of the following four numbers
$$2d(j+1)-d(j+2),\, d(j+1), \, d(j+2), \text{ or } d(j+1)-d(j+2).$$
Therefore, the set of primes that divide $\{S(D_F(n))\}_{n>j}$ is finite and so the set of primes avoided by (\ref{distor}) has density 1.
\end{proof}

Observe that the above proof implies that once $\{S(\sigma_{n,2}+F({\bf X}))\}_{n>j}$ has a zero at position $n_0$, then there are zeros at the positions $n_0+4k$ for $k\in \mathbb{Z}$.   Also, if the zeros in $\{S(\sigma_{n,2}+F({\bf X}))\}_{n>j}$ are removed, then there is only a finite amount of primes that divide the remaining terms.

\begin{example}
Consider the perturbation $\sigma_{n,2}+X_1X_2X_3+X_4X_5X_6$.  Let $d(n)=S(\sigma_{n,2}+X_1X_2X_3+X_4X_5X_6)$.  Observe that 
$$d(7)=-24 \text{ and } d(8)=8.$$
Since $d(8)/d(7)-1 \notin \{0,\pm 1, \infty\}$, then $\sigma_{n,2}+X_1X_2X_3+X_4X_5X_6$ is never balanced.  Moreover, since
$$\{2d(7)-d(8), d(7),d(8), d(7)-d(8)\}=\{-56, -24, 8, -32\},$$
then only the primes 2, 3 and 7 divide some terms of the sequence $\{S(\sigma_{n,2}+X_1X_2X_3+X_4X_5X_6)\}.$  Explicitly, 2 divides every term of the sequence (trivial), 3 divides $d(7+4m)$ and 7 divides  $d(10+4m)$ for every $m\geq 0$.
\end{example}

The story of the case of degree two and its perturbations is short: either $Av(2;F(\bf X))$ is empty or has density 1 because $P(2;F({\bf X}))$ is finite.  The next step (naturally) is the case of degree 3.  At this early stage the study becomes hard.    The next result, which considers the cases for $\sigma_{n,3}$ and $\sigma_{n,[3,2]}$, serves as an example.

\begin{theorem}
The sequences $\{S(\sigma_{n,3})\}$ and $\{S(\sigma_{n,[3,2]})\}$ avoid a prime $p>2$ if and only if $\ord_p(4)\equiv 2 \mod 4$.  In particular, both sequences avoid all primes of the form $8q+5$.
\end{theorem}

\begin{proof}
The proof for $\{S(\sigma_{n,3})\}$  is presented next.  A similar argument can be applied to $\{S(\sigma_{n,[3,2]})\}$.

Recall that $\{S(\sigma_{n,3})\}$ satisfies the recurrence 
\begin{equation}
\label{char3}
x_n = 4x_{n-1}-6x_{n-2}+4 x_{n-3},
\end{equation}
with initial conditions $x_1=2$, $x_2=4$ and $x_3 = 6$.  Finding the roots of the characteristic polynomial and using the initial conditions yield
\begin{equation}
S(\sigma_{n,3}) = 2^{n-1}+2^{n/2}\sin(n \pi/4).
\end{equation}
Thus,
\begin{equation}
\label{closedform}
S(\sigma_{n,3}) = \begin{cases}
 2^{4m-1} & \text{if } n=4m \\
 4^m(4^m+(-1)^m)& \text{if } n=4m+1 \\
 2^{2m+1}(4^m+(-1)^m) & \text{if } n=4m+2 \\
 2^{2m+1}(2^{2m+1}+(-1)^m) & \text{if } n=4m+3.
\end{cases}
\end{equation}
Equation (\ref{closedform}) implies that an odd prime $p$ divides $S(\sigma_{n,3})$ if and only if $p$ divides $4^m+(-1)^m$ for some $m$ or $p$ divides $2^{2m+1}+(-1)^m$ for some $m$.  Thus, the problem is reduced to find primes $p$ for which the congruences
\begin{equation}
\label{firstcong}
4^m+(-1)^m \equiv 0 \mod p
\end{equation}
and 
\begin{equation}
\label{secondcong}
2^{2m+1}+(-1)^m \equiv 0 \mod p
\end{equation}
have solutions.

It is not hard to show that if (\ref{secondcong}) has a solution, then (\ref{firstcong}) has a solution, thus it is enough to find solutions for  (\ref{firstcong}).  According to the parity of $m$, the solutions must be integers $a$ such that
\begin{equation}
\label{finalcong}
4^{2a}\equiv -1\mod p \,\,\, \text{ or } \,\,\, 4^{2a+1}\equiv 1\mod p.
\end{equation}

Suppose $\ord_p(4)=4l$, then $4^{2l}\equiv -1 \mod p$. Thus, there are solutions to (\ref{finalcong}) and therefore $p$ divides $\{S(\sigma_{n,3})\}$.  Also, it is clear that if $\ord_p(4) = 2l+1$, then (\ref{finalcong}) has solutions and therefore $p$ also divides $\{S(\sigma_{n,3})\}$ in this case.  Consider now the last case, this is, suppose that $\ord_p(4)=4l+2$. Then,
\begin{eqnarray*}
4^{4l+2} &\equiv& 1 \mod p\\
4^{2l+1} &\equiv& -1 \mod p.
\end{eqnarray*}
This implies that every solution to $4^j\equiv 1 \mod p$ is even and every solution to $4^j\equiv -1 \mod p$ is odd.  Therefore, (\ref{finalcong}) does not have solutions mod $p$, which implies that $\{S(\sigma_{n,3})\}$ avoids $p$.   This concludes the first statement of the theorem.

Suppose now that $p=8q+5$.  Since $4$ is a square, then
\begin{equation}
4^{\frac{p-1}{2}}=4^{4q+2}\equiv 1 \mod p
\end{equation}
and therefore $\ord_p(4) \, |\, 4q+2$.  Suppose that $\ord_p(4)$ is an odd factor of $4q+2$.  Say $\ord_p(4)=2k+1$.  Then, $4^{2k+1}\equiv 1 \mod p$.  Let $x=2^k$ and observe that
\begin{equation}
(2x^2-1)(2x^2+1)\equiv 0 \mod p.
\end{equation}
But $p$ cannot divide $2x^2-1$ because this would imply that $(p+1)/2$ is a square modulo $p$ and $(p+1)/2$ is a quadratic non-residue for $p\equiv 5 \mod 8$.  Similarly, $p$ does not divide $2x^2+1$ because $(p-1)/2$ is a quadratic non-residue for $p\equiv 5 \mod 8.$    Therefore, $\ord_p(4)$ odd leads to a contradiction.  It must be that $\ord_p(4)$ is even.  Since $\ord_p(4) \, |\, 4q+2$, then $\ord_p(4)$ is divisible by 2, but not by 4.  Therefore $\ord_p(4)\equiv 2\mod 4$ and, by previous argument, $\{S(\sigma_{n,3})\}$ avoids $p$.  This concludes the proof. 
\end{proof}

It is a classical result that the natural density of the set of primes in the arithmetic progression $a n+b$, $\gcd(a,b)=1$, is $1/\varphi(a)$ where $\varphi$ is the Euler's totient function.  The above theorem shows that $Av(3)=Av(3,2)$ contains a subset that has positive natural density, namely $\{p \text{ prime } \,|\, p\equiv 5 \mod 8\}$, which has density 1/4.  This implies that $Av(3)=Av(3,2)$ is infinite.  Observe that if $p\equiv 3 \mod 4$, then $\ord_p(4)$ is odd.  This implies that $\{p \text{ prime} \,|\, p\equiv 3 \mod 4\} \subseteq P(3) = Av(3)^c.$  Therefore, if $Av(3)$ has natural density, then $1/4\leq \delta(Av(3))\leq 1/2$.  Experiments suggest that $\delta(Av(3))$ is close to $1/3.$  Calculating the exact density of a subset of primes is, generally, not easy.  Perhaps an argument similar to the one presented in \cite{hasse} by Hasse is needed to calculate the exact density of $Av(3).$  This problem will not be considered in this work.

Continue with the case of degree 3.  The other two symmetric Boolean functions of degree 3 that have not been considered are $\sigma_{n,[3,1]}$ and $\sigma_{n,[3,2,1]}$.  In these cases, the sets $Av(3,1)$ and $Av(3,2,1)$ are empty because $\{S(\sigma_{n,[3,1]})\}$ and $\{S(\sigma_{n,[3,2,1]})\}$ satisfy (\ref{char3}) with initial conditions $0,0,2$ and $0,-2,-6$, respectively.  Thus, this method cannot be used to identify them as not balanced.  However, observe that these symmetric Boolean functions are asymptotically not balanced.  In fact,
\begin{equation}
\lim_{n\to \infty} \frac{S(\sigma_{n,[3,1]})}{2^n} = \frac{1}{2} \text{ and }  \lim_{n\to \infty} \frac{S(\sigma_{n,[3,2,1]})}{2^n} = -\frac{1}{2}.
\end{equation}  
Solving (\ref{char3}) with the corresponding initial conditions leads to 
\begin{equation}
S(\sigma_{n,[3,1]})=2^{n-1}-2^{n/2} \sin \left(\frac{\pi  n}{4}\right) 
\end{equation}
and
\begin{equation}
S(\sigma_{n,[3,2,1]})=-2^{n-1}+ 2^{n/2} \cos \left(\frac{\pi  n}{4}\right).
\end{equation}
From here it is not hard to see that for all $n\geq 4$,
\begin{equation}
\left|\frac{S(\sigma_{n,[3,1]})}{2^n} - \frac{1}{2}\right|<\frac{1}{3} \text{ and } \left|\frac{S(\sigma_{n,[3,2,1]})}{2^n} + \frac{1}{2}\right|<\frac{1}{3}.
\end{equation}
Thus, $\sigma_{n,[3,1]}$ and $\sigma_{n,[3,2,1]}$ are in fact are not balanced for every $n\geq 3$, even though $\{S(\sigma_{n,[3,1]})\}$ and $\{S(\sigma_{n,[3,2,1]})\}$ are divisible by every prime.

Perturbations of $\sigma_{n,3}$ exhibit a similar behavior.  The next result provides an example.
\begin{proposition}
The set of primes avoided by the sequence 
$$\{S(\sigma_{n,3}+X_1+X_2X_3+X_1X_2X_3)\}_{n\geq 4}$$ is infinite.  In particular, it contains a subset with positive density. 
\end{proposition}

\begin{proof}
Let $F({\bf X})=X_1+X_2X_3+X_1X_2X_3$. The sequence $\{S(\sigma_{n,3}+F({\bf X}))\}_{n\geq 4}$ satisfies the linear recurrence (\ref{char3}) with initial conditions $4,4,-4.$
%\begin{eqnarray*}
%S(\sigma_{4,3}+X_1+X_2X_3+X_1X_2X_3)&=&4\\
%S(\sigma_{5,3}+X_1+X_2X_3+X_1X_2X_3)&=&4\\
%S(\sigma_{6,3}+X_1+X_2X_3+X_1X_2X_3)&=&-4.
%\end{eqnarray*}
Solving the corresponding characteristic equation, using the initial conditions and simplification yield
\begin{equation}
\label{closedformper}
S(\sigma_{n,3}+F({\bf X})) = \begin{cases}
 -2^{2 m-3} \left(2^{2 m}+(-1)^m12 \right) & \text{if } n=4m \\
 -2^{2 m-2} \left(2^{2 m}+(-1)^m8 \right)& \text{if } n=4m+1 \\
 -2^{2 m-1} \left(2^{2 m}+(-1)^m2\right) & \text{if } n=4m+2 \\
 -2^{2 m} \left(2^{2 m}-(-1)^m2 \right) & \text{if } n=4m+3.
\end{cases}
\end{equation}
Thus, the sequence $\{S(\sigma_{n,3}+F({\bf X}))\}_{n\geq 4}$ avoids a prime $p$ if and only if $p$ does not divide any of the following numbers
\begin{equation}
\label{numberstodivide}
2^{2 m}+(-1)^m12,\,\, 2^{2 m}+(-1)^m8,\,\, 2^{2 m}+(-1)^m2,\,\, 2^{2 m}-(-1)^m2
\end{equation}
for every positive integer $m$.

The proof now follows the path of less resistance.  Recall that 
\begin{eqnarray*}
\left(\frac{2}{p}\right)&=&\begin{cases}
\hfill 1 & \text{if } p\equiv 1,7,17,23 \mod 24 \\
\hfill -1& \text{if } p\equiv 5,11,13,19 \mod 24
\end{cases}\\
\left(\frac{-2}{p}\right)&=&\begin{cases}
 \hfill 1 & \text{if } p\equiv 1,11,17,19 \mod 24 \\
 \hfill-1& \text{if } p\equiv 5,7,13,23 \mod 24
\end{cases}\\
\left(\frac{3}{p}\right)&=&\begin{cases}
\hfill 1 & \text{if } p\equiv 1,11,13,23 \mod 24 \\
\hfill -1& \text{if } p\equiv 5,7,17,19 \mod 24
\end{cases}\\
\left(\frac{-3}{p}\right)&=&\begin{cases}
 \hfill 1 & \text{if } p\equiv 1,7,13,19 \mod 24 \\
 \hfill-1& \text{if } p\equiv 5,11,17,23 \mod 24.
\end{cases}
\end{eqnarray*}
Therefore, if $p\equiv 5 \mod 24$, then $\pm 2, \pm 3$ are quadratic non-residues.  This implies that $p$ does not divide any of the numbers in (\ref{numberstodivide}).  Since
\begin{equation}
\delta(\{p \text{ prime}\,|\, p\equiv 5 \mod 24\}) = \frac{1}{\varphi(24)}=\frac{1}{8},
\end{equation}
then the result follows.
\end{proof}

The cases of degree 4 and beyond have not been studied in detail.  However, experimentations lead the authors to believe that they are harder to tackle (as expected).   As an example, consider the 
perturbation $\sigma_{n,4}+F({\bf X})$ where $F({\bf X}) =X_1X_2$.  The set $Av(4;F({\bf X}))$ appears to be infinite.  In fact, it looks like it contains an infinite amount of primes of the form 
$8q+1$.  For instance, the following primes, which are of the form $8q+1$, are members of $Av(4;F({\bf X})):$
$$17, 73, 89, 97, 113, 193, 233, 241, 257, 281, 337, 401, 433, 449.$$
At the moment of writing this manuscript, the authors did not see any pattern on them.

This section gave some insights of the set $Av(k_1,\cdots,k_s)$.  This set is important from the Boolean functions point of view.  In the literature, however, it is common to study the set of primes 
that divide some term of an integer sequence, in other words, the complement of $Av(k_1,\cdots,k_s)$.  Because of this, a short study for the set $P(k_1,\cdots, k_s)$ is presented in the next section.

%%%%%%%%%%%%%%%%%%%%%%%%%%%%%%%%%%%%%%%%%%%%%%%%%%%%%%%%%%%%%%%%%%%%%%%%%%%%
% The set P(k_1,...,k_s)
%%%%%%%%%%%%%%%%%%%%%%%%%%%%%%%%%%%%%%%%%%%%%%%%%%%%%%%%%%%%%%%%%%%%%%%%%%%%
\section{The set $P(k_1,\cdots,k_s)$}
\label{sectionP}

This work finishes with a study of set $P(k_1,\cdots,k_s)$, i.e. the set of primes that divide some term of the sequence $\{S(\sigma_{n,[k_1,\cdots,k_s]})\}$.  It has been stated that P{\'o}lya's result 
does not guarantee that $P(k_1,\cdots,k_s)$ is infinite.  In fact, Theorem \ref{pert2} shows that for some perturbations of the form $\sigma_{n,2}+F({\bf X})$, the set of primes that divide some term of 
the sequence $\{S(\sigma_{n,2}+F({\bf X}))\}$ is finite. The reality is, however, that at least for elementary symmetric Boolean functions this set is infinite.  The idea behind the approach to prove 
this claim is very simple: to identify the odd primes $p$ with the property that $p|S(\sigma_{p,[k_1,\cdots,k_s]})$, i.e. odd primes $p$ that divide the $p$-th term of the sequence (it is not hard to see 
that 2 always divides $S(\sigma_{2,[k_1,\cdots,k_s]})$ and so only odd primes are of interest).  For a general integer sequence, this might not be an easy task, however the representation of 
$S(\sigma_{n,[k_1,\cdots,k_s]})$ as the binomial sum
\begin{equation}
S(\sigma_{n,[k_1,\cdots,k_s]})=\sum_{j=0}^n (-1)^{\binom{j}{k_1}+\cdots+\binom{j}{k_s}}\binom{n}{j}
\end{equation}
simplifies the efforts.  Indeed, if $n=p$ is an odd prime, then 
\begin{equation}
S(\sigma_{p,[k_1,\cdots,k_s]})=\sum_{j=0}^p (-1)^{\binom{j}{k_1}+\cdots+\binom{j}{k_s}}\binom{p}{j}.
\end{equation}
Since $p|\binom{p}{j}$ for every $1\leq j\leq p-1$, then it is clear that $p|S(\sigma_{p,[k_1,\cdots,k_s]})$ if and only if $(-1)^{\binom{p}{k_1}+\cdots+\binom{p}{k_1}}=-1$.
It is a known that the sequence $$\left\{\binom{j}{k_1}+\cdots+\binom{j}{k_s}\mod 2\right\}_{j\geq 0}$$ is periodic.  To be specific, if $r=\lfloor\log_2(k_s)\rfloor+1$, then
\begin{equation}
\binom{j+m\cdot 2^r}{k_1}+\cdots+\binom{j+m\cdot 2^r}{k_s}\equiv \binom{j}{k_1}+\cdots+\binom{j}{k_s} \mod 2
\end{equation}
for every integer $m$.   Let $j_1,\cdots, j_t$ be all integers in $\{0,1,\cdots, 2^r-1\}$ with the property that $\binom{j}{k_1}+\cdots+\binom{j}{k_s}$ is odd.  Assume for the moment that at least one 
of $j_1,\cdots, j_t$ is an odd integer, say $j_o$ is such number.  Every integer of the form $j_o+m\cdot 2^r$ satisfies
\begin{equation}
(-1)^{\binom{j_o+m\cdot 2^r}{k_1}+\cdots+\binom{j_o+m\cdot 2^r}{k_s}}=-1,
\end{equation}
and the classical result of Dirichlet implies that there are an infinite amount of primes of the form $j_o+m\cdot 2^r$.  Of course, if every integer in $j_1,\cdots, j_t$ is even, then every odd prime $p$ satisfies $p\nmid S(\sigma_{p,[k_1,\cdots,k_s]})$ because $S(\sigma_{p,[k_1,\cdots,k_s]})\equiv 2\mod p$.  This information is summarized in the following results.

\begin{theorem}
\label{thmP}
Let $1\leq k_1<\cdots<k_s$ be integers.  Let $j_1,\cdots, j_t$ be all integers in $\{0,1,\cdots, 2^r-1\}$ with the property that $$\binom{j}{k_1}+\cdots+\binom{j}{k_s}$$ is odd.  Then, there is at least one odd prime $p$ such that $p|S(\sigma_{p,[k_1,\cdots,k_s]})$  if and only if at least one of the integers $j_1,\cdots, j_t$ is odd.  Moreover, if $j_{o_1},\cdots,j_{o_f}$ are the odd integers in  $j_1,\cdots, j_t$, then the odd primes $p$ with the property $p|S(\sigma_{p,[k_1,\cdots,k_s]})$ are precisely the primes in the arithmetic progressions
$$j_{o_1}+m\cdot 2^r,\cdots, j_{o_f}+m\cdot 2^r,$$
where $r=\lfloor\log_2(k_s)\rfloor +1$. 
\end{theorem}
\begin{proof}
This is a direct consequence of the previous discussion.
\end{proof}

\begin{corollary}
Let $1\leq k_1<\cdots<k_s$ be integers.  Let $j_1,\cdots, j_t$ be all integers in $\{0,1,\cdots, 2^r-1\}$ with the property that $$\binom{j}{k_1}+\cdots+\binom{j}{k_s}$$ is odd.   Let $b(k_1,\cdots,k_s)$ the number of odd integers in $j_1,\cdots, j_t$  and assume that $b(k_1,\cdots,k_s)>0$.  Then $P(k_1,\cdots,k_s)$ is an infinite set because it contains a subset of density $$\frac{b(k_1,\cdots,k_s)}{2^{\lfloor \log_2(k_s) \rfloor}}$$
within the set of primes.
\end{corollary}
\begin{proof}
This is a consequence of the statement of Theorem \ref{thmP} and the classical result of Dirichlet.
\end{proof}

In the case of the elementary symmetric Boolean function of degree $k$, the set $P(k)$ is infinite.

\begin{corollary}
Suppose that $k$ is a natural number.  The set $P(k)$ is infinite. 
\end{corollary}

\begin{proof}
The statement is clearly true if $k$ is a power of 2, therefore suppose that $k\neq 2^l$ for some integer $l$.  The idea is to identify the set of primes in Theorem \ref{thmP} for such $k$.  

Let $k=2^{a_1}+\cdots+2^{a_e}$ be the 2-adic expansion of $k$.   Lucas' Theorem implies that $j$ is such that $\binom{j}{k}$ is odd if and only if 
\begin{equation}
\label{jfork}
j=k+\sum_{2^i \notin\{2^{a_1},\cdots,2^{a_e}\}} \delta_i 2^i
\end{equation}
where $\delta_i \in \{0,1\}$.  By Theorem \ref{thmP}, one must identify the odd $j$'s of the form (\ref{jfork}) that lie in the set $\{1,2,3,\cdots, 2^r-1\}$, where $r=\lfloor \log_2(k) \rfloor+1$.  
If $k$ is odd, it is not hard to see that there are $2^{r-w_2(k)}$ of such $j$, where $w_2(n)$ represents the sum of the binary digits of $n$.  If $k$ is even, then the total is $2^{r-w_2(k+1)}$.  
In either case, the set $P(k)$ is infinite.  Moreover, making the transformation $\bar{k}=2\lfloor k/2\rfloor+1$, one sees that $P(k)$ is infinite because it contains a subset of density 
$1/2^{w_2(\bar{k})-1}$ within the set of primes.
\end{proof}

\begin{example}
Consider the case $k_1=3, k_2=4, k_3=7$ and $k_4=9$.  The integers $j$'s in $\{1,2,\cdots, 15\}$ for which
\begin{equation}
\binom{j}{3}+\binom{j}{4}+\binom{j}{7}+\binom{j}{9}
\end{equation}
is odd are $3, 4, 5, 6, 7, 9, 12,$ and  $14$.  Theorem \ref{thmP} implies that all primes $p$ in the arithmetic progressions 
$$3+ 16m, \,\,5+16m,\,\, 7+16m, \,\,9+16m$$
satisfy $p|S(p,[3,4,7,9])$.  Therefore, $P(3,4,7,9)$ is infinite because it contains a subset of density $$4/2^{\lfloor \log_2(9) \rfloor}=4/8=1/2$$ within the set of all primes.
\end{example}

Theorem \ref{thmP} identifies the condition on the degrees $k_1,\cdots, k_s$ for the existence of odd primes $p$ with the property $p|S(p,[k_1,\cdots,k_s])$.  A natural question to consider is to identify 
the symmetric Boolean functions for which Theorem \ref{thmP} fails to apply.  These are symmetric Boolean functions $\sigma_{n,[k_1,\cdots,k_s]}$ with the property that every integer $j$ for which 
\begin{equation}
\binom{j}{k_1}+\cdots+\binom{j}{k_s}
\end{equation}
is odd is an even integer.  An example of such Boolean function is $\sigma_{n,[4,5]}$ and the reason the theorem fails is very simple to explain. Lucas' Theorem implies that every time $\binom{j}{5}$ is odd, so is $\binom{j}{4}$.  This means that $\binom{j}{4}+\binom{j}{5}$ is odd precisely when $\binom{j}{4}$ is odd and $\binom{j}{5}$ is even, but this occurs when $j=4+8m$ or $j=6+8m$ for $m$ non-negative integer.  The same argument applies to $\sigma_{n,[2,3]}$, $\sigma_{n,[6,7]}$ and every Boolean function of the form $\sigma_{n,[2l,2l+1]}.$  Also, any combination of them have the same property, for example,  $p\nmid S(\sigma_{p,[2,3,6,7]})$ for every odd prime.  The reader can convince himself/herself that Theorem \ref{thmP} does not apply to a symmetric Boolean function if and only if it is a combination of terms of the form $\sigma_{n,[2l,2l+1]}$.  

Theorem \ref{thmP} provides a sufficient condition for $P(k_1,\cdots, k_s)$ to be infinite and the above discussion characterizes all symmetric Boolean functions for which Theorem \ref{thmP} fails to apply.  The amount of these ``bad" symmetric Boolean functions increases exponentially as the degree grows.  To be specific, if all symmetric Boolean functions of degree $k_s\leq M$ are considered, then $2^{\lfloor (M-1)/2 \rfloor}-1$ of them are such that Theorem \ref{thmP} fails to apply.  However, it is also true that this number is exponentially smaller than the total amount of all symmetric Boolean functions of degree less than or equal to $M$.  For example, consider all symmetric Boolean functions of degree less than or equal to 15.  There is a total of $32767$ of such Boolean functions, but only 127 of them are such that $p\nmid S(\sigma_{p,[k_1,\cdots,k_s]})$ for every odd prime $p$.  Note that
\begin{equation}
1-\frac{127}{32767}\approx 0.9961241493,
\end{equation}
which means that Theorem \ref{thmP} can be used to identify $P(k_1,\cdots,k_s)$ as an infinite set for 99.61\% of them.  Of course, the other 0.39\% of the sets $P(k_1,\cdots,k_s)$ may still be infinite.  For example, consider the case of $\sigma_{n,[3,2]}$.  By Theorem \ref{thmP}, $p\nmid S(\sigma_{p,[3,2]})$ for every odd prime $p$, yet the set $P(3,2)$ is infinite.  The reason for this is that $p\in Av(3,2)$ if and only if $\ord_p(4)\equiv 2 \mod 4.$  However, if $p\equiv 3 \mod 4$, then $\ord_p(4)$ is odd and so $p\in P(3,2)$, which clearly implies that $P(3,2)$ is an infinite set.

\medskip \medskip
\noindent
{\bf Acknowledgments.} The authors would like to thank Professor Thomas W. Cusick for reading a previous version of this article.  His comments and suggestions improve the presentation of this work.
The second author acknowledges the partial support of UPR-FIPI 1890015.00. \\

%\begin{conjecture}
%Suppose that $1\leq k_1<\cdots<k_s$.  Then $\{S(\sigma_{n,[k_1,\cdots,k_s]})\}$ cannot avoid $11$, i.e. $11\in P(k_1,\cdots,k_s)$.
%\end{conjecture}

%\end{appendices}

%%%%%%%%%%%%%%%%%%%%%%%%%%%%%%%%%%%%%%%%%%%%%%%%%%%%%%%%%%%%%%%%%%%%%
% References
%%%%%%%%%%%%%%%%%%%%%%%%%%%%%%%%%%%%%%%%%%%%%%%%%%%%%%%%%%%%%%%%%%%%%
\bibliographystyle{plain}

\vfill

\appendix
%\part*{Appendixes}
\section{Table containing $\mathfrak{p}(k)$ for $5\leq k\leq 263$}
\label{app:table}

%\begin{appendices}
%\renewcommand\thetable{\thesection\arabic{table}}
% \renewcommand\thefigure{\thesection\arabic{figure}}
%\section{Table} 
%\label{app:table}
$$
\begin{array}{|c|c|c|c|c|c|c|c|c|c|c|c|}
\hline
k&  5 & 6 & 7 & 9 & 10 & 11 & 12 & 13 & 14 & 15 & 17 \\
\mathfrak{p}(k) & 17 & 41 & 73 & 97 & 17 & 17 & 17 & 17 & 17 & 17 & 1601 \\
\hline
k&18 & 19 & 20 & 21 & 22 & 23 & 24 & 25 & 26 & 27 & 28 \\
\mathfrak{p}(k) & 97 & 97 & 97 & 449 & 257 & 97 & 97 & 97 & 97 & 193 & 257 \\
\hline
k& 29 & 30 & 31 & 33 & 34 & 35 & 36 & 37 & 38 & 39 & 40 \\
\mathfrak{p}(k) & 97 & 97 & 97 & 449 & 193 & 1409 & 193 & 193 & 193 & 257 & 193 \\
\hline
k& 41 & 42 & 43 & 44 & 45 & 46 & 47 & 48 & 49 & 50 & 51 \\
\mathfrak{p}(k) & 449 & 769 & 257 & 193 & 449 & 257 & 193 & 193 & 193 & 193 & 257 \\
\hline
k& 52 & 53 & 54 & 55 & 56 & 57 & 58 & 59 & 60 & 61 & 62 \\
\mathfrak{p}(k) & 449 & 193 & 193 & 193 & 257 & 449 & 257 & 257 & 257 & 449 & 641 \\
\hline
k& 63 & 65 & 66 & 67 & 68 & 69 & 70 & 71 & 72 & 73 & 74 \\
\mathfrak{p}(k) & 193 & 1153 & 257 & 769 & 641 & 257 & 769 & 641 & 257 & 641 & 769 \\
\hline
k& 75 & 76 & 77 & 78 & 79 & 80 & 81 & 82 & 83 & 84 & 85 \\
\mathfrak{p}(k) & 1409 & 257 & 641 & 257 & 257 & 257 & 641 & 257 & 641 & 257 & 257 \\
\hline
k& 86 & 87 & 88 & 89 & 90 & 91 & 92 & 93 & 94 & 95 & 96 \\
\mathfrak{p}(k) & 257 & 3329 & 2689 & 2689 & 641 & 257 & 257 & 257 & 257 & 257 & 257 \\
\hline
k& 97 & 98 & 99 & 100 & 101 & 102 & 103 & 104 & 105 & 106 & 107 \\
\mathfrak{p}(k) & 769 & 769 & 257 & 769 & 641 & 1153 & 257 & 1153 & 257 & 769 & 641 \\
\hline
k& 108 & 109 & 110 & 111 & 112 & 113 & 114 & 115 & 116 & 117 & 118 \\
\mathfrak{p}(k) & 257 & 3329 & 257 & 257 & 257 & 641 & 3329 & 257 & 257 & 641 & 257 \\
\hline
k& 119 & 120 & 121 & 122 & 123 & 124 & 125 & 126 & 127 & 129 & 130 \\
\mathfrak{p}(k) & 641 & 257 & 257 & 257 & 769 & 257 & 1153 & 257 & 641 & 769 & 769 \\
\hline
k& 131 & 132 & 133 & 134 & 135 & 136 & 137 & 138 & 139 & 140 & 141 \\
\mathfrak{p}(k) & 257 & 769 & 3329 & 3329 & 257 & 7681 & 7937 & 257 & 257 & 3329 & 769 \\
\hline
k& 142 & 143 & 144 & 145 & 146 & 147 & 148 & 149 & 150 & 151 & 152 \\
\mathfrak{p}(k) & 257 & 257 & 257 & 7681 & 257 & 257 & 7937 & 257 & 3329 & 257 & 257 \\
 \hline
k& 153 & 154 & 155 & 156 & 157 & 158 & 159 & 160 & 161 & 162 & 163 \\
\mathfrak{p}(k) & 7681 & 257 & 257 & 257 & 257 & 257 & 257 & 257 & 3329 & 9473 & 769 \\
  \hline
k& 164 & 165 & 166 & 167 & 168 & 169 & 170 & 171 & 172 & 173 & 174 \\
\mathfrak{p}(k) &  257 & 7937 & 3329 & 257 & 769 & 257 & 257 & 257 & 257 & 257 & 257 \\
  \hline
k& 175 & 176 & 177 & 178 & 179 & 180 & 181 & 182 & 183 & 184 & 185 \\
\mathfrak{p}(k) & 257 & 7937 & 3329 & 257 & 257 & 257 & 257 & 257 & 769 & 257 & 257 \\
 \hline
k&  186 & 187 & 188 & 189 & 190 & 191 & 192 & 193 & 194 & 195 & 196 \\
\mathfrak{p}(k) & 257 & 257 & 257 & 257 & 257 & 257 & 257 & 769 & 257 & 257 & 257 \\
 \hline
k&  197 & 198 & 199 & 200 & 201 & 202 & 203 & 204 & 205 & 206 & 207 \\
\mathfrak{p}(k) & 7937 & 10753 & 257 & 257 & 769 & 257 & 257 & 257 & 257 & 257 & 257 \\
 \hline
k&  208 & 209 & 210 & 211 & 212 & 213 & 214 & 215 & 216 & 217 & 218 \\
\mathfrak{p}(k) & 257 & 3329 & 257 & 257 & 257 & 257 & 257 & 9473 & 257 & 257 & 257 \\
 \hline 
k&  219 & 220 & 221 & 222 & 223 & 224 & 225 & 226 & 227 & 228 & 229 \\
\mathfrak{p}(k) & 257 & 257 & 257 & 257 & 257 & 257 & 3329 & 257 & 257 & 257 & 769 \\
 \hline 
 k&  230 & 231 & 232 & 233 & 234 & 235 & 236 & 237 & 238 & 239 & 240 \\
\mathfrak{p}(k) & 257 & 257 & 257 & 257 & 257 & 257 & 257 & 257 & 257 & 257 & 257 \\
 \hline 
 k&  241 & 242 & 243 & 244 & 245 & 246 & 247 & 248 & 249 & 250 & 251 \\
 \mathfrak{p}(k) & 7937 & 769 & 257 & 257 & 257 & 257 & 257 & 257 & 257 & 257 & 257 \\
 \hline 
 k&  252 & 253 & 254 & 255 & 257 &258  & 259 & 260  & 261 & 262 & 263  \\
 \mathfrak{p}(k) & 257 & 257& 257 & 257 & 7681 & 7681  &  7681 &  15361 & 19457  & 7681 & 10753  \\
 \hline

\end{array}
$$

\end{document}